\documentclass[twoside]{amsart}
\usepackage{graphicx}
\usepackage{epsfig}
\usepackage{float}

\newtheorem{theorem}{Theorem}[section]
\newtheorem{lemma}[theorem]{Lemma}
\theoremstyle{definition}

\theoremstyle{remark}
\newtheorem{remark}{\bf Remark}[section]
\numberwithin{equation}{section}
\newcommand{\R}{I\!\!R}
\newcommand{\bw}{{\bf w}}
\newcommand{\bq}{{\bf q}}
\newcommand{\bW}{{\bf W}}

\newcommand{\bt}{\boldsymbol\tau}
\newcommand{\bs}{\boldsymbol\sigma}

\newcommand{\eu}{\eta_u}
\newcommand{\eq}{\boldsymbol\eta_{\bf q}}
\newcommand{\es}{\boldsymbol\eta_{\boldsymbol\sigma}}
\newcommand{\xu}{\xi_u}
\newcommand{\xq}{\boldsymbol\xi_{\bf q}}
\newcommand{\xs}{\boldsymbol\xi_{\boldsymbol\sigma}}

\newcommand{\hxu}{\hat{\xi}_u}
\newcommand{\hxq}{\hat{\boldsymbol\xi}_{\bf q}}
\newcommand{\hxs}{\hat{\boldsymbol\xi}_{\boldsymbol\sigma}}

\newcommand{\imh}{{i-1/2}}
\newcommand{\iph}{{i+1/2}}
\newcommand{\xUh}{\hat{\xi}_U}
\newcommand{\xQh}{\hat{\boldsymbol\xi}_\Q}
\newcommand{\xZh}{\hat{\boldsymbol\xi}_\Z}

\newcommand{\eU}{\eta_U}
\newcommand{\eQ}{\boldsymbol\eta_\Q}
\newcommand{\eZ}{\boldsymbol\eta_\Z}
\newcommand{\xU}{\xi_U}
\newcommand{\xQ}{\boldsymbol\xi_\Q}
\newcommand{\xZ}{\boldsymbol\xi_\Z}

\newcommand{\ta}{\tau^\ast}

\newcommand{\A}{\mathcal {A}}
\newcommand{\B}{\mathcal {B}}

\newcommand{\ii}{{\boldsymbol \epsilon}}
\newcommand{\hii}{\hat{\boldsymbol \epsilon}}
\newcommand{\Bnj}{{\tilde B}^\nph_\jph}

\newcommand{\ph}{{\partial}_t}
\newcommand{\pc}{{\delta}_t}

\newcommand{\pb}{\bar{\partial}_t}
\newcommand{\ps}{{\partial}_{t}^2}
\newcommand{\half}{{1/2}}

\newcommand{\Z}{{\boldsymbol Z}}
\newcommand{\Q}{{\boldsymbol Q}}
\newcommand{\U}{{U}}
\newcommand{\dt}{k}
\newcommand{\qr}{1/4}

\newcommand{\nph}{{n+1/2}}
\newcommand{\msph}{{m^{\star}+1/2}}
\newcommand{\jph}{{j+1/2}}
\newcommand{\nmh}{{n-1/2}}

\newcommand{\nq}{{n;1/4}}
\newcommand{\ah}{A^\half}

\newcommand{\ce}{{\mathcal E}_{\B}}
\newcommand{\hce}{\hat{\mathcal E}_{\B}}

\newcommand{\sQ}{R}
\newcommand{\sZ}{S}

\begin{document}
\title[An $hp$-DGM for Linear Hyperbolic Integro-Differential Equations]{ \bf A  Priori $hp$-estimates for  Discontinuous Galerkin Approximations to Linear Hyperbolic Integro-Differential Equations  }
\author{Samir Karaa}
\address{Department of Mathematics and Statistics, Sultan Qaboos University, P. O. Box 36, Al-Khod 123,
Muscat, Oman}
\email{skaraa@squ.edu.om}
\author{   Amiya K. Pani}
\address{Department of Mathematics, Industrial Mathematics Group, Indian
Institute of Technology Bombay, Powai, Mumbai-400076}
\email{akp@math.iitb.ac.in}
\author{ Sangita Yadav}
\address{Department of Mathematics,
Birla Institute of Technology and Science Pilani,
Pilani, Rajasthan-333031}
\email{sangita.iitk@gmail.com}
\date{}
\maketitle


\begin{abstract}
An $hp$-discontinuous Galerkin (DG) method is applied to a class of
second order linear hyperbolic integro-differential
equations. Based on the analysis of an expanded mixed type
Ritz-Volterra projection, {\it a priori} $hp$-error
estimates in  $L^{\infty}(L^2)$-norm of the velocity
as well as of the displacement, which are optimal in the discretizing parameter
$h$ and suboptimal in the degree of polynomial $p$ are derived. For optimal estimates of the displacement in $L^{\infty}(L^2)$-norm with reduced regularity on the exact solution, a variant of Baker's nonstandard energy formulation is developed and analyzed. Results on order of convergence which are similar in spirit to linear elliptic and parabolic problems are established for the semidiscrete case after suitably modifying the
numerical fluxes. For the completely discrete scheme, an implicit-in-time procedure is formulated, stability results are derived and {\it a priori} error estimates are discussed. Finally, numerical experiments on two dimensional domains are
conducted which confirm the theoretical results.
\end{abstract}

{\bf Keywords} - {\small Local discontinuous Galerkin method, linear second order hyperbolic integro-differential
equation, nonstandard formulation, semidiscrete and completely discrete
schemes, mixed type Ritz-Volterra projection, role of stabilizing parameters,
$hp$-error estimates, order of convergence, numerical experiments.}

\section{Introduction}
In this paper, we discuss discontinuous Galerkin (DG) methods which include the local discontinuous Galerkin method (LDG)  for the following second order linear hyperbolic integro-differential equation:
\begin {eqnarray}
\quad u_{tt}-\nabla\cdot\Big(A(x)\nabla u +\int_{0}^{t} B(x,t,s) \nabla u(s) \,ds\Big)
={f}(x,t) &&~~~\mbox{in}~ \Omega \times (0~ T],\label{1} \\
u(x,t)= 0\hspace{0.8cm}&&~~~\mbox{on}~\partial \Omega\times (0~T],\label{2}\\
u_{|t=0}= u_0 \hspace{.6cm}&&~~~\mbox{in}~\Omega,\label{3}\\
u_{t|t=0}= u_1 \hspace{.6cm}&&~~~\mbox{in}~\Omega,\label{4}
\end{eqnarray}
where $u_{tt}=({\partial^2 u}/{\partial t^2}),$ and $f$, $u_0,\, u_1$ are given functions. We assume that $\Omega$ is a bounded convex domain in
$\R^2$ with boundary $\partial\Omega$, $ {A(x)}=[a_{ij}(x)]$ is a $2\times 2$ positive
definite matrix such that there exists a positive constant $\alpha$ with
$ (A(x)\xi,\xi)\geq \alpha |\xi|^2,\, 0\neq \xi \in \R^2$ and ${B(x,t,s)}=[b_{ij}(x,t,s)]$
is a $2\times 2$ matrix. Further, assume that all coefficients of
 $A$ and $B$ are smooth and bounded functions with bounded derivatives in their respective
 domain of definitions say by a positive constant $M$. Such classes of problems and nonlinear version, thereof,  arise naturally
in many applications, such as, in viscoelasticity, see \cite{RHN} and reference, therein.

In literature, finite Galerkin methods are applied to  hyperbolic
integro- differential equations and  {\it a priori} $h$-error estimates have been extensively studied for the
problem (\ref{1})-(\ref{4})
by Cannon {\it et al.}~\cite{CLX}, Pani {\it et. al.}~\cite{PSO}-\cite{PTW}, Lin {\it et
al.}~\cite{LTW}, Sihna \cite{S}, Sinha and Pani \cite{SP}, Yanik and Fairweather~\cite{YF}.

Of late, there has been a lot of activities in discontinuous Galerkin (DG) methods for
approximating solutions of partial differential equations. This  is mainly due to their
flexibility in local mesh adaptivity and in taking care of  nonuniform degrees of
approximation of the solution whose smoothness may exhibit wide variation  over the
computational domain. Like finite volume element methods,  these methods  are element-wise
conservative, but are ideally suited to $hp$-adaptivity. For application of DG methods to elliptic and parabolic problems, we may refer Cockburn {\it et al.} \cite{cks} for review of development of DG methods. One such DG method is the
local discontinuous Galerkin (LDG) method which allows for arbitrary meshes with hanging nodes,
elements of various shapes and piecewise polynomials of varying degrees. Earlier, Cockburn
and Shu~\cite{cs} have introduced this method for convection-diffusion problems and
subsequently, $hp$-version error estimates are derived by Castillo {\it et al.}~\cite{ccss}.
The LDG method was then extended to elliptic problems by Cockburn {\it et al.}~\cite{ccps},
Perugia and Sch$\ddot o$tzau~\cite{ps} and Gudi {\it et al.}~\cite{gnp}. In
\cite{ccps}, optimal order of convergence of LDG method applied to a Poisson equation
has been established. Subsequently, Perugia and Sch$\ddot o$tzau~\cite{ps} have discussed  {\it a priori} $hp$-error estimates for linear elliptic problems
and then,  Gudi {\it et al.}~\cite{gnp} have derived $hp$-error
estimates  for nonlinear elliptic problems. For higher order partial differential equations using
LDG method, see \cite{chen-shu,cs,lsy,xs,xs1} and references, therein. 

In this paper, $hp$-DG methods which, in particular, include the original LDG scheme,
are applied to the problem (\ref{1})-(\ref{4}). Further, it is observed that if polynomials of degree at least $p$ are used in all the
elements, the rates of convergence in the $L^{\infty}(L^2)$-norm of
the displacement $u$ and its velocity $\bq=\nabla u$ are of order $p+1/2$ and $p$,
respectively, provided the stabilization parameter $C_{11} =O(1)$ with $C_{22}=0.$
When $C_{11}=O(1/h),$ it is shown that the order of convergence of $u$ is  $p+1.$
Based on expanded mixed type Ritz-Volterra projection as an intermediate solution, optimal estimates
are derived. Using a variant of Baker's  nonstandard formulation, {\it a priori} estimates in $L^{\infty}(L^2)$
norm for the displacement are established with reduced regularity conditions on the exact solution.
All the above results are proved for semi-discrete method. Then, an implicit-in-time method is applied to the semi-discrete  scheme to provide a completely discrete method and stability results are proved. Again a use of a modified Baker's argument combined with a more finner analysis to take care of integral term yields
{\it a priori} estimates  for the displacement in $\ell^{\infty}{L^2}$-norm.
Finally, some numerical experiments for the LDG method have been performed
with different degrees of polynomials and numerical results are presented
to support the theoretical results. Our $hp$-estimates are  also valid for second order wave equations only by making $B\equiv 0.$ For applications of other DG methods to wave equations, we may refer
to \cite{rw1}-\cite{rw2}, \cite{gss}-\cite{gs}.

Throughout this paper, we denote $C$ as a generic positive constant which does
not depend on the discretizing parameter $h$ and degree of polynomial $p,$ but may vary
from time to time.

The article is organized as follows. Section 2 deals with  preliminaries
and basic results to be used subsequently in the rest of the article. In Section 3, we formulate
DG methods. Section 4 is devoted to an extended mixed type Ritz-Volterra
projection and related estimates. In Section 5, we discuss {\it a priori} error estimates for the semidiscrete scheme. Section 6 focuses on the completely discrete scheme based on an implicit method and related
error estimates are derived. In Section 7,
some numerical experiments are conducted to confirm the theoretical results.

 \section{Preliminaries}
Let $ \mathcal{T}_h=\{K_i:1\leq i\leq N_h\},\,\, 0< h< 1$ be a family of triangulation consisting of shape regular finite elements, which decompose the domain $\Omega$ into a finite number of simplexes $K_i,$ where $K_i$ is either a triangle or rectangle. It is, further, assumed that the family of triangulations  satisfies bounded local variation conditions on mesh size and on polynomial degree.
Let $h_i$ be the diameter of $ K_i$ and $ h= \max\{ h_i:1 \leq i\leq N_h\}$.
 We denote the set of interior edges of $\mathcal{T}_h $ by $ \Gamma_I=\{e_{ij}:e_{ij}=
\partial K_i\cap \partial K_j,~ |e_{ij}| > 0\}$ and boundary edges by $\Gamma_
\partial=\{e_{i\partial}:e_{i\partial}=\partial K_i\cap
\partial \Omega,~|e_{i\partial}|>0\} $, where $|e_k|$ denotes the one dimensional
 Euclidean measure. Let $ \Gamma=\Gamma_I\cup\Gamma_\partial$.
Note that our definition of $e_k$ also includes hanging nodes along each side of
 the finite elements. On this subdivision $\mathcal{T}_h$, we define the following broken Sobolev spaces
$$ V=\{v\in L^2(\Omega): v_{|_{K_i}} \in H^1(K_i),~ \forall~ K_i\in \mathcal{T}_h\},$$
and
$${\bf{W}} = \{{\bf{w}} \in {\bf L}^2(\Omega): {\bf{w}}_{ |_{K_i}}\in {\bf H}^1(K_i), ~\forall~ K_i \in \mathcal{T}_h\},$$
where $H^1(K_i)$ is the standard Sobolev space of order one defined on $K_i$, ${\bf L}^2(\Omega)=(L^2(\Omega))^2$ and ${\bf H}^1(K_i)=(H^1(K_i))^2$. The associated broken norm and
 seminorm on $V$ are defined, respectively, as
$$ \|v\|_{H^1(\mathcal{T}_h)}= \left ( \sum_{i=1}^{N_h}\|v\|^2_{H^1(K_i)} \right)^{\frac{1}{2}} \mbox{and}~~ |v|_{H^1(\mathcal{T}_h)}=\left(\displaystyle\sum_{i=1}^{N_h} |v|^{2}_{H^1(K_i)}\right)^\frac{1}{2}.$$\\
We denote the $L^2$-inner product by $(\cdot,\cdot)$ and the norm by $\|\cdot\|$. We also use broken Sobolev spaces
$$H^r(\mathcal{T}_h)=\{v\in L^2(\Omega):\sum_{i=1}^{N_h}\|v\|^2_{H^r(K_i)}<\infty\},$$ with norm $\|v\|_{H^r(\mathcal{T}_h)}=\Big(\displaystyle\sum_{i=1}^{N_h}\|v\|^2_{H^r(K_i)}\Big)^\frac{1}{2}$.\\
Further, we define for a Hilbert space $X$
$$L^p(0,T;X)=\{\phi:[0, T]\rightarrow X:\int_0^T\|\phi(t)\|^p_X\,dt<\infty\},$$
with norm for $1\le p<\infty$
$$\|\phi\|_{L^p(0,T;X)}=\Big(\int_0^T\|\phi(t)\|^p_X\,dt\Big)^\frac{1}{p},$$
and for $p=\infty$,
$$\|\phi\|_{L^\infty(0,T;X)}=\displaystyle\mbox{ess sup}_{t\in(0,T)} \|\phi(t)\|_X.$$
For notational convenience, we denote $L^p(0,T;X)$ as $L^p(X)$.

Let $e_k\in \Gamma_I$, that is $e_k=\partial K_i\cap\partial K_j$ for some neighboring simplexes  $K_i$ and $K_j$. Let ${\boldsymbol\nu}_i$ and ${\boldsymbol\nu}_j$
be the outward normals to the boundary $\partial K_i$ and $\partial K_j$, respectively. On $e_k$, we now define
the jump and average of $v\in V$ as
$$ [\![v]\!]=v_{|_{K_i}}{\boldsymbol\nu}_i+v_{|_{K_j}}{\boldsymbol\nu}_j,\hspace{.5cm} \{\!\!\{v\}\!\!\}=\frac{v_{|_{K_i}}+v_{|_{K_j}}}{2},$$
respectively, and for $\bf{w}\in \bf{W}$, the jump and average are defined as
$$[\![\bw]\!]=\bw_{|_{K_i}}\cdot{\boldsymbol\nu}_i+\bw_{|_{K_j}}\cdot{\boldsymbol\nu}_j,\hspace{.5cm}\{\!\!{\{\bf w }\}\!\!\}=\frac{\bw_{|_
{K_i}}+\bw_{|_{K_j}}}{2}.$$
In case, $e_k\in \partial\Omega$, that is, there exists $K_i$ such that $e_k=\partial K_i\cap
\partial\Omega$, then set the jump and average for $v$ as
$$ [\![v]\!]=v_{|_{K_i\cap\partial\Omega}}\boldsymbol \nu,\hspace{.5cm} \{\!\!\{v\}\!\!\}=v_{|_{K_i\cap\partial\Omega}},$$
respectively, and for $\bw\in{\bW}$, the jump and average are defined respectively by
$$ [\![\bw]\!]=\bw_{|_{K_i\cap\partial\Omega}}\cdot\boldsymbol\nu,\hspace{.5cm} \{\!\!\{\bw\}\!\!\}=\bw_{|_{K_i\cap\partial
\Omega}},$$
where $\boldsymbol\nu$ is the outward normal to the boundary $\partial\Omega$. Let $P_{p_i}(K_i)$ be the
 space of polynomials of degree less than or equal to $p_i$ on each triangle $K_i\in \mathcal{T}_h$ and $Q_{p_{i}}(K_i)$
be the space of polynomials of degree less than or equal to $p_i$ in
each variable which are defined on the rectangles $K_i\in \mathcal{T}_h$. The discontinuous
finite element spaces are considered as
$$ V_h=\{v_h\in L^2(\Omega):{v_h}_{|_{K_i}}\in Z_{p_i}(K_i)\},$$
and
$$ {\bW}_h=\{\bw_h\in{\bf L}^2(\Omega):{\bw_h}_{|_{K_i}}\in {\bf Z}_{p_i}(K_i)\},$$
 where ${\bf Z}_{p_i}(K_i)=(Z_{p_i}(K_i))^2, $ $p_i\geq 1$ and $Z_{p_i}(K_i)$ is either $P_{p_i}(K_i)$ or $Q_{p_{i}}(K_i)$. For any $e_k
 \in \Gamma_I $, there are two elements $K_i$ and $K_j$ such that $e_k=\partial K_i\cap \partial K_j$.
 We associate $p_k$ to $e_k$ where $p_k=\frac{p_i+p_j}{2}$.
 For $e_k \in \Gamma_\partial$, since there is one element $K_i$ such that $e_k=\partial K_i\cap \partial
  \Omega$, we write $p_k=p_i$.
We also denote $ p=\displaystyle\min_{1\leq i\leq N_h}{p_i}.$

Below, we state a Lemma without proof on the approximation properties of the finite
element spaces. For a  proof, refer to \cite{BS}.
\begin{lemma}\label{Ih-proj}
 For $\phi\in (H^{r_i}(K_i))^d,~d=1,~2$, there exist a positive constant
$C_A$ depending on $r_i,$ but independent of $\phi, p_i$ and $h_i$
 and a sequence $\phi_{p_i}^h\in (Z_{p_i}(K_i))^d, p_i\geq 1$, such that
\begin{itemize}
\item [(i)] for any $0\leq l\leq r_i$,
$$ \|\phi-\phi^h_{p_i}\|_{(H^l(K_i))^d}\leq C_A\frac{h_i^{\min\{r_i,p_i+1\} -l}}
{p_i^{r_i-l}}\|\phi\|_{(H^{r_i}(K_i))^d}, $$
\item [(ii)] for $  r_i > l+\frac{1}{2},$
$$
\|\phi-\phi^h_{p_i}\|_{(H^l(e_k))^d}\leq
C_A\frac{h_i^{\min\{r_i,p_i+1\}-l-{\frac{1}{2}}}}{p_i^{r_i-l-{\frac{1}{2}}}}
\|\phi\|_{(H^{r_i}(K_i))^d}.
$$
\end{itemize}
 \end{lemma}
 For any $\phi \in {\bf W},$ we define ${\bf I}_h \phi \in {\bf W}_h$ by
 $$
{\bf I}_h \phi|_{K_i} = \phi^{h}_{p_i},\;\;\; {\mbox{ for }} \; K_i \in
{\mathcal{T}}_h.
$$
We observe that ${\bf I}_h$ satisfies the local approximation properties  given
in Lemma \ref{Ih-proj}. In a similar manner, we can also define
$I_h \psi,$ for $\psi \in V.$

\begin{lemma}({\bf $L^2$-projection $\Pi_h$}).\label{L2-proj}
Let $\boldsymbol\psi\in {\bf H}^{r_i+1}(K_i)$ and
$ \boldsymbol\psi_h=\Pi_h\boldsymbol\psi \in {\bf Z}_{p_i}(K_i) $ be the
$L^2$-projection of
 $\boldsymbol\psi$ onto ${\bf Z}_{p_i}(K_i)$. Then the following
approximation property holds:
$$ \| \boldsymbol\psi- \boldsymbol\psi_h\|_{(L^2(K_i))^2}
+\frac {h_i^{\frac{1}{2}}}{p_i}\| \boldsymbol\psi-
\boldsymbol\psi_h\|_{(L^2(\partial K_i))^2}
\leq C\frac{h_i^{\min(r_i,~ p_i)+1}}{p_i^{r_i+1}}
\| \boldsymbol\psi\|_{{\bf H}^{r_i+1}(K_i)}.$$
\end{lemma}

 \section { Discontinuous Galerkin Method}

In order to formulate DG methods for hyperbolic integro-differential
equations (\ref{1})-(\ref{4}), we now introduce the gradient and flux
variables as
\begin{eqnarray*}
 \bq=\nabla u,\hspace{.5cm}{\boldsymbol \sigma}=A\bq+\int_0^t B(t,s)\bq \,ds,
\end{eqnarray*}
and then rewrite (\ref{1}) as a system of equations:
\begin{eqnarray}
\bq&=&\nabla u \;\;\;\;\;\;\mbox{in}~\Omega\label{21},\\
{\boldsymbol \sigma}&=& A\bq+\int_0^t B(t,s)\bq \,ds \;\;\;\mbox{in}~
\Omega\label{22},\\
u_{tt}-\nabla\cdot{\boldsymbol \sigma}&=& f \;\;\;\;\;\;
\mbox{in}~\Omega\label{23}.
\end{eqnarray}
Then, the DG formulation for (\ref{21})-(\ref{23}) is to seek an approximate
solution $(u_h,{\bf q_h},{\boldsymbol \sigma}_h): (0,T]\mapsto Z_p(K)\times {\bf Z}_p(K)\times {\bf Z}_p(K)$   satisfying  for all $K\in\mathcal{T}_h$, the following system of equations for all $(v_h,\bw_h,{\boldsymbol \tau}_h) \in Z_p(K)\times {\bf Z}_p(K)\times {\bf Z}_p(K)$:
\begin{eqnarray}
\int_K{\bq_h}\cdot\bw_hdx+\int_K u_h \nabla\cdot\bw_hdx-\int_{\partial K}\hat{u}\bw_h\cdot \boldsymbol\nu_K \,ds=0,\label{ldg1}\\
\int_K A\bq_h\cdot\bt_h\,dx -\int_K{\bs}_h\cdot\bt_h\,dx +\int_0^t \int_K B(t,s)\bq_h(s)\cdot\bt_h\,dx\,ds=0,\label{ldg2}\\
\int_K u_{htt} v_h \,dx+\int_K {\boldsymbol \sigma}_h\cdot\nabla v_h \,dx-\int_{\partial K}\hat{{\boldsymbol \sigma}}\cdot\boldsymbol\nu_K v_h~\,ds=\int_K{ f} v_h \,dx.\label{ldg3}
\end{eqnarray}
Here, the numerical fluxes $\hat{u}$ and $\hat{\bs}$  are defined on
$ e_k \in \Gamma_I,$ see \cite{ccps},  as:
\begin{eqnarray}
\hat{u}(u_h,\bs_h)&=&\{\!\!\{u_h\}\!\!\}+C_{12}\cdot[\![u_h]\!]
-C_{22}[\![\bs_h]\!],\label{f1}\\
\hat{\bs}(u_h,{\boldsymbol \sigma}_h)&=&\{\!\!\{\bs_h\}\!\!\}
-C_{11}[\![u_h]\!]-C_{12}[\![\bs_h]\!],\label{f2}
\end{eqnarray}
and for  $e_k\in \Gamma_\partial$, i.e., $e_k=\partial K\cap\partial\Omega$ for
some $K\in \mathcal T_h$, then the numerical fluxes are denoted by
\begin{eqnarray}
 \hat u&=&0\\
 \hat\bs&=&{\bs_h}_{|_K}-C_{11} {u_h}_{|_K}{\boldsymbol\nu}{ _K},
\end{eqnarray}
where the parameters $C_{11}$, $C_{12}\in {\rm I{\!}R}^2 $ and $C_{22}$ are
single valued and are to be chosen appropriately. It is observed that
the numerical fluxes are conservative and consistent (cf. \cite{abcm}).

To complete the  DG formulation, sum (\ref{ldg1})-(\ref{ldg3}) over all elements $K\in \mathcal{T}_h$ and
apply  the conservative property  and the definition of the numerical
fluxes to arrive at the following system of equations for all
$(v_h,\bw_h,{\boldsymbol\tau}_h)\in V_h\times{\bW}_h\times{\bW}_h$:
\begin{eqnarray*}
&&\hspace{-.8cm}\int_\Omega{\bq_h}\cdot\bw_hdx+\sum_{i=1}^{N_h}\int_{K_i}\!\! u_h \nabla\cdot\bw_hdx-\!\!\int_{\Gamma_I}\!\!(\{\!\!\{u_h\}\!\!\}+C_{12}\cdot[\![u_h]\!]-C_{22}[\![\bs_h]\!])[\![\bw_h]\!]\,ds\!=\!0,\label{ldg21}\hspace{0cm}\\
&&\hspace{-.8cm}\int_\Omega A\bq_h.{\boldsymbol \tau}_h\,dx -\int_\Omega{\bs}_h.{\boldsymbol \tau}_h\,dx+\int_0^t \int_\Omega B(t,s)\bq_h(s)\cdot
{\boldsymbol \tau}_h\,dx\,ds=0,\label{ldg22}\\
&&\hspace{-.72cm}\label{ldg23}\int_{\Omega}u_{htt} v_h \,dx+\sum_{i=1}^{N_h}\int_{K_i} {\boldsymbol \sigma}_h\cdot\nabla v_h \,dx-\int_{\Gamma}(\{\!\!\{\boldsymbol \sigma_h\}\!\!\}-C_{11}[\![u_h]\!]\!-C_{12}[\![\bs_h]\!])[\![v_h]\!]\,ds\\&&\hspace{8cm}=\int_\Omega f v_h \,dx.\nonumber
\end{eqnarray*}
Note that the LDG method is obtained, when $C_{22}=0$, that is, when the numerical flux $\hat u$ does not
depend on $\bs_h.$

To rewrite the above system in a compact form, we define the following bilinear and linear forms:

\noindent
$ \A:{\bW}\times{\bW}\rightarrow\R $ as
\begin{eqnarray*}
\A({\bf p},\bw)=\int_\Omega {\bf p}\cdot\bw \,dx,
\end{eqnarray*}
$\A_1:V\times{\bW}\rightarrow \R$ as
\begin{eqnarray*}
\A_1(v,{\bf p})&=&\sum_{i=1}^{N_h}\int_{K_i} {\bf p}\cdot\nabla v \,dx-\int_{\Gamma}(\{\!\!\{{\bf p}\}\!\!\}-C_{12}[\![{\bf p}]\!])[\![v]\!]\,ds,\\
                &=&-\sum_{i=1}^{N_h}\int_{K_i} v \nabla\cdot{\bf p}+\int_{\Gamma_I}(\{\!\!\{v\}\!\!\}+C_{12}\cdot[\![v]\!])[\![{\bf p}]\!]\,ds,
\end{eqnarray*}
$ \A_2:{\bW}\times{\bW}\rightarrow\R $ as
$$ \A_2({\bf p},\bw)=\int_\Omega  A(x) {\bf p}\cdot\bw~\,dx,$$
$\B:{\bW}\times{\bW}\rightarrow \R$ as
$$ \B(t,s;{\bf p}(s),\bw)=\int_\Omega B(t,s){\bf p}(s)\cdot\bw~\,dx,$$
$J_1:{\bW}\times{\bW}\rightarrow \R $ as
$$J_1({\bf p},\bw)=\int_{\Gamma_I}C_{22}[\![{\bf p}]\!][\![\bw]\!]ds,$$
and
$ J: V\times V\rightarrow \R$ as
$$J(\phi,v)=\int_{\Gamma}C_{11}[\![\phi]\!][\![v]\!] ds.$$
Hence, the DG formulation for the problem (\ref{21})-(\ref{23})
in compact form is stated as: find
$(u_h,\bq_h,{\boldsymbol \sigma}_h):(0,T]\rightarrow
V_h\times{\bW}_h\times{\bW}_h$ such that
\begin{eqnarray}
\A(\bq_h,\bw_h)-\A_1(u_h,\bw_h)+J_1(\bs_h, \bw_h)=0,\;\forall \bw_h \in {\bW}_h,\label{ldg31}\\
\; \A_2(\bq_h,{\boldsymbol \tau}_h)-\A({\boldsymbol \sigma}_h,{\boldsymbol \tau}_h)+
\int_0^t \B(t,s;\bq_h(s),{\boldsymbol \tau}_h)\,ds=0,\;\forall {\bt}_h\in {\bW}_h,  \label{ldg32}\\
(u_{htt},v_h)+\A_1(v_h,{\boldsymbol \sigma}_h)+J(u_h,v_h)=(f, v_h),\;\forall v_h\in V_h.\label{ldg33}
\end{eqnarray}

Following \cite{PY-2011}, we are now ready to specify the stabilization parameters. We define
 the set $\left<K,K'\right>$ by
\begin{equation}
 \left<K,K'\right>=\left\{\begin{array}{l l}\emptyset~~~~~~~~~~~~~~~~~~~~~~~~~~~~~~~~~~ \mbox {if  meas}(\partial K\cap\partial K')=0,\\
                          \mbox {interior of}~ \partial K\cap\partial K'~~~~~~~~~~~~~~~~~~~~~~~~~~ \mbox{otherwise}.
                         \end{array}\right.
\end{equation}
Assume that the stabilization parameters $C_{11}$ and $C_{22}$
in the definition of  numerical fluxes in (\ref{f1}) and (\ref{f2}) are stated, respectively, as
\begin{eqnarray}\label{c11}
 C_{11}({\boldsymbol x})=\left\{\begin{array}{l l}
\zeta\min\{\frac{h_i^\alpha}{ p_i^{2\alpha}},
\frac{h_j^\alpha}{ p_j^{2\alpha}}\}&\mbox {if} ~{\boldsymbol x}
\in \left< K_i,K_j\right>,\\
 \zeta \frac{h_i^\alpha}{ p_i^{2\alpha}};&\mbox {if} ~{\boldsymbol x} \in \partial K_i\cap\partial\Omega,\end{array}\right.
 \end{eqnarray}
 and
 \begin{eqnarray}\label{c22}
 C_{22}({\boldsymbol x})=\left\{\begin{array}{l l}\kappa\min\{\frac{h_i^\beta}{ p_i^{2\beta}},\frac{h_j^\beta}{ p_j^{2\beta}}\}&\mbox {if} ~{\boldsymbol x} \in \left< K_i,K_j\right>,\\
 \kappa \frac{h_i^\beta}{ p_i^{2\beta}}&\mbox {if} ~{\boldsymbol x} \in \partial K_i\cap\partial\Omega,\end{array}\right.
 \end{eqnarray}
where  $\zeta>0$, $\kappa\ge 0$, $-1\le\alpha\le0\le\beta\le 1$ are independent
of mesh size and $|C_{12}|$ is of order one. Our main results will be written
in terms of the parameters $\mu^*$ and $\mu_*,$
$$\mu^*=\max\{-\alpha,\hat\beta\},~~\mu_*=\min\{-\alpha,\hat\beta\},$$
where $\hat\beta=1,$ if $\kappa=0 $ and $\hat\beta=\beta,$ otherwise.

For each edge, we define
\begin{equation}
 \Lambda(\boldsymbol x):=\left\{\begin{array}{l l}
\min\{\frac{h_i}{p_i^2},\frac{h_{j}}{p_{j}^2}\}
& \mbox{if}~\boldsymbol x\in\left<K_i,K_j\right>,\\
~\frac{h_i}{p_i^2}& \mbox{if}
~\boldsymbol x\in\partial K_i\cap\partial\Omega,\end{array}\right.
\end{equation}
and then, we set
\begin{equation}\label{chi}
 \chi(\boldsymbol x):=\left\{\begin{array}{l l}
\Lambda(\boldsymbol x)&\hspace{3cm}\mbox{if}~\kappa=0,\\
 C_{22}(\boldsymbol x)&\hspace{3cm}\mbox{otherwise}.\end{array}\right.
\end{equation}

Note that numerical fluxes $\hat{u}$ and $\hat{\boldsymbol \sigma}$ are consistent and hence,
we  arrive at  the following system of equations for all $(v_h,{\boldsymbol
\tau}_h,\bw_h)\in V_h\times{\bW}_h\times{\bW}_h$,
\begin{eqnarray}
&& \A(\bq-\bq_h,\bw_h)-\A_1(u-u_h,\bw_h)+J_1(\bs-\bs_h, \bw_h)=0,\label{er1}\\
&& \A_2(\bq-\bq_h,{\boldsymbol  \tau}_h)-\A(\boldsymbol{\sigma- \sigma}_h,{\boldsymbol \tau}_h)+
\int_0^t \B(t,s;(\bq-\bq_h)(s),{\boldsymbol \tau}_h)\,ds=0,\label{er2}\\
&&(u_{tt-}u_{htt},v_h)+ \A_1(v_h,\boldsymbol{\sigma-\sigma}_h)+J(u-u_h,v_h)=0.\label{er3}
\end{eqnarray}
Below, we state  two theorems  on convergence  of the semidiscrete scheme, whose proof
can be found in the end of  Section 5.
\begin{theorem}\label{thm:main}
Let $(u,\bq,\bs)$ be the solution of $(\ref{21})$-$(\ref{23})$ satisfying
$u\in L^\infty(H^{r+2}(\mathcal{T}_h))$ with $u_{tt}\in L^1(H^{r+2}(\mathcal{T}_h))$
for $r\geq 0$. Further, let $(u_h,\bq_h,\bs_h)\in V_h\times{\bW}_h\times  {\bW}_h$ be the
solution of $(\ref{ldg31})$-$(\ref{ldg33})$. If $u_h(0)=\Pi_hu_0$, $u_{ht}(0)=\Pi_hu_1$,
$q_h(0)={\bf I}_h\nabla u_0$ and $\bs_h(0)={\bf I}_h(A\nabla u_0)$,  then the following estimates hold:
\begin{eqnarray}\label{eu2}
\|u_t-u_{ht}\|_{L^\infty(L^2(\Omega))}
&\leq & C \frac{h^{P+D}}{p^{\sQ+\sZ}}\Big(\|u_0\|_{H^{r+2}(\mathcal{T}_h)} +\|u_1\|_{H^{r+1}(\mathcal{T}_h)} \\
&&+\sum_{j=0}^{2}\left\|\frac{\partial^{j}u }{\partial t^j}\right\|_{L^1(H^{r+2}(\mathcal{T}_h))}\Big),\nonumber
\end{eqnarray}
and
\begin{eqnarray}\label{eq-es}
\|{\bf q-q}_h\|_{L^\infty(L^2(\Omega))}&+&\|\bs-\bs_h\|_{L^\infty(L^2(\Omega))}
\leq  C \frac{h^{P}}{p^{R}}\Big(\|u_0\|_{H^{r+2}(\mathcal{T}_h)} +\|u_1\|_{H^{r+2}(\mathcal{T}_h)} \nonumber\\
&+&\sum_{j=0}^{2}\left\|\frac{\partial^{j}u }{\partial t^j}\right\|_{L^1(H^{r+2}(\mathcal{T}_h))} \Big),
\end{eqnarray}
where $P=\min\{r+\frac{1}{2}(1+\mu_*),p+\frac{1}{2}(1-\mu^*)\},~D=\frac{1}{2}(1+\mu_*)$,
$\sQ=r+\min\{\mu_*,1-\mu^*\}$ and $\sZ=\min\{\frac{1}{2},\mu_*\}$.
\end{theorem}
\begin{theorem}\label{FE-2}Let $\Omega$ be a bounded convex polygon domain
in $\R^2$ and let $(u,\bq,\bs)$ be the solution of $(\ref{21})$-$(\ref{23})$ satisfying
$u\in L^\infty(H^{r+2}(\mathcal{T}_h))$ with $u_t\in L^1(H^{r+2}(\mathcal{T}_h))$ for $r\geq0$.
Further, let $(u_h,\bq_h,\bs_h)\in V_h\times{\bW}_h\times  {\bW}_h$ be the
solution of $(\ref{ldg31})$-$(\ref{ldg33})$ with $u_h(0)=\Pi_hu_0.$ Then the following estimate holds:
\begin{eqnarray}
\quad \|u-u_h\|_{L^{\infty}(L^2(\Omega))} &\leq & C \frac{h^{P+D}}{p^{\sQ+\sZ}}\Big(\|u_0\|_{H^{r+2}(\mathcal{T}_h)}
+\sum_{j=0}^{1}\left\|\frac{\partial^{j}u }{\partial t^j}\right\|_{L^1(H^{r+2}(\mathcal{T}_h))}\Big),
\label{eu3}
\end{eqnarray}
where $P,D,R$ and $S$ as in Theorem~$\ref{thm:main}$.
\end{theorem}
\begin{table}[H]
\caption{}
\vspace{-0.4em}
\begin{center}
{Orders of convergence for $r\ge p$ and $p\ge1$.}
\begin{tabular}{|c c|c|c|}
\hline
& & & \\
$C_{22}$&$C_{11}$ & ~~~$\|\bq-\bq_h\|,~\|\bs-\bs_h\|$ &~~~~~$\|u-u_h \|$\\
& & & \\
\hline
& & & \\
0, $O(\frac{h}{p^2})$&$O(1)$   & $\frac{h^{p}}{p^r}$ & $\frac{h^{p+\frac{1}{2}}}{p^r}$ \\
& & & \\
~0, $O(\frac{h}{p^2})$&$O(\frac{p^2}{h})$ & $\frac{h^{p}}{p^{{r}}}$           & $\frac{h^{p+1}}{p^{{r+\frac{1}{2}}}}$		     \\
& & & \\
\hline
& & & \\
$O(1)$&$O(1)$   & $\frac{h^{p+\frac{1}{2}}}{p^r}$ & $\frac{h^{p+1}}{p^r}$ \\
& & & \\
$O(1)$&$O(\frac{p^2}{h})$& $\frac{h^{p}}{p^r}$ & $\frac{h^{p+\frac{1}{2}}}{p^r}$ \\
& & & \\
\hline
\end{tabular}

\label{OC}
\end{center}
\end{table}
As in \cite{PY-2011}, we can make similar observations based on the results of
the above Theorem~\ref{thm:main}.

\section{Extended Mixed Ritz-Volterra Projection and
Related Estimates}

In this section, we introduce an extended mixed Ritz-Volterra
projection for our subsequent use.

Define an extended mixed Ritz-Volterra projection as:
Find $(\tilde{u}_h,\tilde{\bf q}_h,\tilde{\boldsymbol
\sigma}_h):(0,T]\rightarrow V_h\times{\bW}_h\times{\bW}_h $ satisfying
\begin{eqnarray}
&& \A(\bq-{\bf \tilde{q}}_h,\bw_h)- \A_1(u-\tilde{u}_h,\bw_h)+J_1(\bs-\tilde{\bs}_h, \bw_h)=0,~~~\forall~\bw_h\in\bW_h,\label{p1}\\
&&\A_2(\bq-{\bf \tilde{q}}_h,{\bf \tau}_h)
-\A(\boldsymbol{\sigma-\tilde{\sigma}}_h,{\bt}_h)
+\int_0^t \B(t,s;(\bq-{\bf \tilde{q}}_h)(s),{\bt}_h)\,ds
=0,~~~\forall~\bt_h\in\bW_h,\label{p2}\hspace{-1cm}\\
&& \A_1(v_h,\boldsymbol{\sigma- \tilde{\sigma}}_h)
+J(u-\tilde{u}_h,v_h)=0,~~~\forall~v_h\in V_h.\label{p3}
\end{eqnarray}
For given $(u,\bq, \bs)$, it is easy to show the  existence of a unique  solution  $(\tilde{u}_h,\tilde{q}_h,\tilde{\boldsymbol \sigma}_h)$ to the problem (\ref{p1})-(\ref{p3}).

With $ \eu:= u-\tilde u_h,\;\; \eq:={\bf q-\tilde q}_h$, and $\es:=\bs-\tilde{\bs}_h,$  we state without proof
the error estimates, whose proofs after simple modifications can be found in \cite{PY-2011}.
\begin{theorem} \label{thm:u-up}
Let  $(u,\bq,\bs)$ be the solution
of $(\ref{21})$-$(\ref{23})$.
Further, let $(\tilde u_h,\tilde\bq_h,\tilde
{\bs}_h)$ be the solution of $(\ref{p1})$-$(\ref{p3})$. Then, there exists a positive
constant $C$ independent of $h$ and $p$ such that for $l=0,1,2$
\begin{eqnarray*} \label{etau2}
\left|\left|\frac{\partial^l \eu}{\partial t^l}\right|\right|\leq  C\frac{h^{(P+D)}}{p^{(\sQ+\sZ)}}
\sum_{j=0}^{l} \left(\left\|\frac{\partial^j u}{\partial t^j}\right\|_{H^{r+2}
(\mathcal{T}_h)}+\int_0^t\left\|\frac{\partial^{j}u}{\partial t^j}(s)\right\|_{H^{r+2}
(\mathcal{T}_h)}\,ds\right),\nonumber
\end{eqnarray*}
and
\begin{eqnarray}
\left\|\frac{\partial^l \es}{\partial t^l}\right\|+
\left\|\frac{\partial^l \eq}{\partial t^l}\right\|
+\left(\int_{\Gamma_I} C_{22}\left[\!\left[\frac{\partial^l \es}{\partial t^l}\right]\!\right]^2\,ds\right)^\half
+\left(\int_\Gamma C_{11}\left[\!\left[\frac{\partial^l \eu}{\partial t^l}\right]\!\right]^2\,ds\right)^\half\nonumber\\
\leq C\frac{h^{P}}{p^{R}}\sum_{j=0}^{l}\left(\left\|\frac{\partial^j u}{\partial t^j}\right\|_{H^{r+2}
(\mathcal{T}_h)}+\int_0^t\left\|\frac{\partial^j u}{\partial t^j}(s)\right\|_{H^{r+2}
(\mathcal{T}_h)}\,ds\right),\hspace{-1cm}\label{tstq}
\end{eqnarray}
where $P=\min\{r+\frac{1}{2}(1+\mu_*),p+\frac{1}{2}(1-\mu^*)\},~D=\frac{1}{2}(1+\mu_*)$,
$\sQ=r+\min\{\mu_*,1-\mu^*\}$ and $\sZ=\min\{\frac{1}{2},\mu_*\}$.
\end{theorem}

\section{ A Priori Error Estimates}

In this section, we shall derive {\it a priori} error estimates for the semi-discrete DG scheme.
Using the extended mixed Ritz-Volterra projection, we rewrite
\begin{eqnarray*}
u-u_h:=(u-\tilde u_h)-(u_h-\tilde u_h)=:\eu-\xu,\\
          {\bf q- q}_h:=({\bf q-\tilde q}_h)-(\bq_h-{\bf\tilde q}_h)=:\eq-\xq,\\
          \bs-\bs_h:=(\bs-\tilde{\bs}_h)-(\bs_h-\tilde{\bs}_h)=:\es-\xs.
\end{eqnarray*}
\noindent
Since the estimates of $\eu$, $\eq$ and $\es$ are known from Theorem \ref{thm:u-up},
it is enough to obtain  estimates for $\xu$, $\xq$ and ${\boldsymbol\xi}_{\bs}.$
Now, from (\ref{p1})-(\ref{p3}) and (\ref{er1})-(\ref{er3}),  we arrive at
\begin{eqnarray}
&& \A(\xq, \bw_h)-\A_1(\xu,\bw_h)+J_1(\xs,\bw_h)=0~~~~~~~~~~~\forall~ \bw_h\in {\bW}_h, \label{xi1}\\
&& \A_2(\xq,{\bt}_h)-\A(\xs, {\bt}_h )+\int_0^t \B(t,s;
\xq(s),{\bt}_h)\,ds=0~~~~~~~~~~~\forall ~{\bt}_h\in {\bW}_h, \label{xi2}\\
&&( \xi_{u_{tt}}, v_h)+ \A_1(v_h,\xs)+J( \xu,v_h)=(\eta_{u_{tt}}, v_h)~~~\forall ~ v_h\in V_h. \label{xi3}
\hspace{-2cm}
\end{eqnarray}
Estimates of $\|\xu\|,~\|\xq\|$ and $\|\xs\|$ are given in the following lemma.
\begin{lemma} \label{5a}
There exists a constant C, independent of $h$ and $p,$ such that
\begin{eqnarray}
&&\|\xi_{u_t}\| + ||\xq ||+\|\xs\| \leq C \Big(\|\xi_{u_t}(0)\|\label{xi_u}+||\xq(0)||\qquad\qquad\\
&&\qquad  +J_1(\xs(0),\xs(0))^\half+J( \xu(0),\xi_{u}(0))^\half+\int_0^T||\eta_{u_{tt}}||\,ds\Big).\nonumber
\hspace{-2.5cm}
\end{eqnarray}
\end{lemma}

\begin{proof}
We differentiate equation (\ref{xi1}) with respect to $t$ and choose $\bw_h=\xs$ in (\ref{xi1}),
${\bt}_h=\boldsymbol\xi_{\bq_t}$ in (\ref{xi2}) and $v_h=\xi_{u_t}$ in
(\ref{xi3}). By adding, we  obtain
\begin{eqnarray} \label{estimate-xit}
\frac{1}{2}\frac{d}{dt}\Big(\|\xi_{u_t}\|^2 &+& A_2(\xq,\xq)+
J_1(\xs,\xs)+J( \xu,\xi_{u})\Big)\\
&=&(\eta_{u_{tt}}, \xi_{u_t})-\int_0^t\B(t,s;\xq(s),\boldsymbol\xi_{\bq_t})\,ds. \nonumber
\end{eqnarray}
Next, we write the integral term on the right hand side of (\ref{estimate-xit}) as
\begin{eqnarray} \label{integral-term}
\int_0^t \B(t,s;\xq(s),\boldsymbol\xi_{\bq_t})\,ds &=&
\frac{d}{dt}\int_0^t\B(t,s;\xq(s),\xq)\,ds - \B(t,t,\xq(t),\xq) \\
&&-\int_0^t\B_t(t,s;\xq(s),\xq)\,ds. \nonumber
\end{eqnarray}
Substitute (\ref{integral-term}) in (\ref{estimate-xit}) and integrate from $0$ to $t.$
Then using the Cauchy-Schwarz inequality,
the boundedness of $B,$ the positive-definite property of $A$ and setting
$$\|(\xi_{u},\xi_{u_t},\xq,\xs)(t)\|^2=\|\xi_{u_t}(t)\|^2+||\ah\xq (t)||^2
+J_1(\xs(t),\xs(t))+J( \xu(t),\xi_{u}(t)),
$$
we arrive at
\begin{eqnarray}
||(\xi_{u},\xi_{u_t},\xq,\xs)(t)||^2&\leq& ||(\xi_{u},\xi_{u_t},\xq,\xs)(0)||^2
+2\int_0^t||\eta_{u_{tt}}||\,||\xi_{u_t}||ds\label{xi2b} \nonumber\\
&&+C(M,\alpha,T)\left(\int_0^t||\ah\xq(s)||\,||\ah\xq(t))||ds\right.\\
&&+\left.\int_0^t||\ah\xq(s)||^2ds\right).\nonumber
\end{eqnarray}
For some $t^\ast\in[0,t]$, let
$$
\|(\xi_{u},\xi_{u_t},\xq,\xs)(t^\ast)\|=\max_{0\leq s\leq t}\|(\xi_{u},\xi_{u_t},\xq,\xs)(s)\|.
$$
Then, at $t=t^\ast$, (\ref{xi2b}) becomes
\begin{eqnarray*}
\|(\xi_{u},\xi_{u_t},\xq,\xs)(t^\ast)\|&\leq& \|(\xi_{u},\xi_{u_t},\xq,\xs)(0)\|
+2\int_0^{t^\ast}\|\eta_{u_{tt}}\|\;ds\\
&&+C(M,\alpha,T)\int_0^{t^\ast}\|\ah\xq(s)\|\;ds,
\end{eqnarray*}
and hence,
\begin{eqnarray*}
\|(\xi_{u},\xi_{u_t},\xq,\xs)(t)\|&\leq& \|(\xi_{u},\xi_{u_t},\xq,\xs)(t^\ast)\|\\
&\leq& \|(\xi_{u},\xi_{u_t},\xq,\xs)(0)\|
+2\int_0^T\|\eta_{u_{tt}}\|\;ds\\
&&+C(M,\alpha,T)\int_0^T\|\ah\xq(s)\|\;ds.
\end{eqnarray*}
Now an application of  Gronwall lemma shows that
\begin{eqnarray} \label{estimate-xit2}
\|\xi_{u_t}\|+||\ah\xq ||
&\leq&C\Big(\|\xi_{u_t}(0)\|+||\xq (0)||+J_1(\xs(0),\xs(0))^\half\label{xi2-a}\\
&&+J( \xu(0),\xi_{u}(0))^\half+\int_0^T||\eta_{u_{tt}}||ds\Big).\nonumber
\end{eqnarray}
To estimate $\|\xs\|$, we choose $\bt_h=\xs$ in (\ref{xi2}) and use the
Cauchy-Schwarz inequality to arrive at
$$
\|\xs\|\leq C\Big(\|\ah\xq\|+\int_0^T\|\ah\xq(s)\|\,ds\Big).
$$
A use of  (\ref{estimate-xit2}) completes the proof of the lemma.
\end{proof}

\noindent
 {\bf Proof of Theorem~\ref{thm:main}}. Using the triangle inequality, we can write
$$
 \|u_t-u_{ht}\|\le\|u_t-\tilde u_{ht}\|+\|\tilde u_{ht}-u_{ht}\|.
 $$
Now a use of Theorem~\ref{thm:u-up} and Lemma~\ref{5a} with the choices
$\tilde u_h(0)=\Pi_hu_0$, $\tilde u_{ht}(0)=\Pi_hu_1$,
and $\tilde\bq_h(0)={\bf I}_h\nabla u_0$ yields the estimate (\ref{eu2}). In the similar way,
we can find the estimate (\ref{eq-es}). This completes the rest of the proof. \hfill{$\Box$}

\begin{remark}
As a consequence of Lemma~\ref{5a} and the following inequality
\begin{eqnarray*}
\|\xi_{u}(t)\|\leq C\Big(\|\xi_{u}(0)\|+ \int_0^t\|\xi_{u_t}\|\,ds\Big),
\end{eqnarray*}
we now obtain an estimate of $\|\xi_{u}\|.$ This, in turn, provides the following $L^{\infty}(L^2)$ estimate
of $u-u_h$ as
\begin{eqnarray}\label{eu2-1}
\|u-u_{h}\|_{L^\infty(L^2(\Omega))}
&\leq & C \frac{h^{P+D}}{p^{\sQ+\sZ}}\Big(\|u_0\|_{H^{r+2}(\mathcal{T}_h)} +\|u_1\|_{H^{r+1}(\mathcal{T}_h)} \\
&&+\sum_{j=0}^{1}\left\|\frac{\partial^{j}u }{\partial t^j}\right\|_{L^1(H^{r+2}(\mathcal{T}_h))}
+ \|u_{tt}\|_{L^1(H^{r+2}(\mathcal{T}_h))}\Big).\nonumber
\end{eqnarray}
\end{remark}

Note that as a consequence of Theorem \ref{thm:main}, we obtain   estimates  (\ref{eu2-1})
under the assumption of higher regularity result on the solution.
We now use a variant of Baker's nonstandard formulation (see \cite{Baker-3}) to
provide a proof of $L^{\infty}(L^2)$ estimate under reduced regularity result.

Now define the function $\hat{\phi}$ by
$$
\hat{\phi}(t)=\int_0^t\phi(s)ds.
$$
After integrating (\ref{xi2}) and (\ref{xi3}) with respect to $t$, we obtain the new system
\begin{eqnarray}
&&\A(\xq, \bw_h)-\A_1(\xu,\bw_h)+J_1(\xs,\bw_h)=0~~~~~~~~~~~\forall~ \bw_h\in {\bW}_h, \label{xi1-n}\\
&&\;\A_2(\hxq,{\bt}_h)-\A(\hxs, {\bt}_h )+\int_0^t\Big(\int_0^s \B(s,\tau;
\xq(\tau),{\bt}_h)d\tau\Big)\,ds=0~~~~~~~~~~~\forall ~{\bt}_h\in {\bW}_h, \label{xi2-n}\\
&&\;( \xi_{u_{t}}, v_h)+\A_1(v_h,\hxs)+J( \hxu,v_h)=(\eta_{u_{t}}, v_h)-(e_{ht}(0), v_h)~~~\forall ~ v_h\in V_h. \label{xi3-n}
\end{eqnarray}
Note that with
$u_{ht}(0)= \Pi_h u_1,$ we have
$$ (e_{ht}(0), v_h)=0,~~~\forall ~ v_h\in V_h.$$

\noindent
{\bf Proof of Theorem~\ref{FE-2}}.
Choose  $\bw_h=\hxs$ in (\ref{xi1-n}), ${\bt}_h=\xq$ in (\ref{xi2-n}) and $v_h=\xi_{u}$ in
(\ref{xi3-n}). Then adding the resulting equations, we find that
\begin{eqnarray*}
&&\frac{1}{2}\frac{d}{dt}\left[\|\xi_{u}\|^2+A_2(\hxq,\hxq)+
J_1(\hxs,\hxs)+J( \hxu,\hxu)\right]\\
&&\qquad \qquad=(\eta_{u_{t}}, \xu)-
\int_0^t\Big(\int_0^s\B(s,\tau;\xq(\tau),\xq)d\tau\Big)\,ds.
\end{eqnarray*}
Integrating from $0$ to $t,$ and using the non-negativity
of $J$ and $J_1$, we arrive at
\begin{eqnarray}\label{xiu-1}
\|\xi_{u}\|^2&+& \|A^\half\hxq\|^2= \|\xi_{u}(0)\|^2+2\int_0^t(\eta_{u_{t}},\xu)\,ds\\
&&-2\int_0^t\int_0^s \int_0^\tau \B(\tau,\ta;\xq(\ta),\xq(s))\;d\ta \;d\tau \,ds.\nonumber
\end{eqnarray}
Let $I$ denote the last term on the right hand side of (\ref{xiu-1}).   Integration by parts yields
\begin{eqnarray*}
&&\int_0^s \int_0^\tau \B(\tau,\ta;\xq(\ta),\xq(s))\;d\ta \;d\tau =
\int_0^s \B(\tau,\tau;\hxq(\tau),\xq(s))\;d\tau \\
&&\quad-\int_0^s \int_0^\tau \B_{\ta}(\tau,\ta;\hxq(\ta),\xq(s))\,d\ta \,d\tau,
\end{eqnarray*}
and therefore $I=-2(I_1-I_2)$ where $$I_1=\int_0^t\int_0^s \B(\tau,\tau;\hxq(\tau),\xq(s))d\tau ds,$$
and
$$I_2=\int_0^t\int_0^s \int_0^\tau \B_{\ta}(\tau,\ta;\hxq(\ta),\xq(s))d\ta d\tau ds.$$
Again, we integrate by parts so that
\begin{eqnarray*}
I_1&=&\int_0^t \B(s,s;\hxq(s),\hxq(t))ds-\int_0^t \B(s,s;\hxq(s),\hxq(s))ds\\
&\leq & M\left\{\|\hxq(t)\|\int_0^t\|\hxq(s)\|ds+ \int_0^t\|\hxq(s)\|^2ds\right\}.
\end{eqnarray*}
Similarly, we have
\begin{eqnarray*}
I_2&=& \int_0^t\int_0^s \B_{\tau}(s,\tau;\hxq(\tau),\hxq(t))d\tau ds
-\int_0^t\int_0^s \B_{\tau}(s,\tau;\hxq(\tau),\hxq(s))d\tau ds\\
&\leq & MT \left\{\|\hxq(t)\| \int_0^t\|\hxq(s)\|ds+ \int_0^t\|\hxq(s)\|^2ds\right\}.
\end{eqnarray*}
Using the Cauchy-Schwarz inequality and the bounds for $I_1$ and $I_2$, we obtain
\begin{eqnarray*}\label{xiu-2}
\|\xi_{u}(t)\|^2+\|A^\half\hxq(t)\|^2\label{xi2-a-n}&\leq& \|\xi_{u}(0)\|^2+2\int_0^t
\|\eta_{u_{t}}(s)\|\|\xi_{u}(s)\|ds\\
&&+C(M,\alpha,T)\left(\|A^\half\hxq(t)\|\int_0^t\|A^\half\hxq(s)\|ds\right.\nonumber\\
&&+ \left.\int_0^t\|A^\half\hxq(s)\|^2ds\right).
\nonumber
\end{eqnarray*}
Now, let $|||(\xi_{u},\hxq)(t)|||^2=\|\xi_{u}(t)\|^2+\|A^\half\hxq(t)\|^2$ and
$$|||(\xi_{u},\hxq)(t^\ast)|||=\max_{0\leq s\leq t}|||(\xi_{u},\hxq)(t)|||,$$
for some $t^\ast \in [0,t]$. Then, at $t=t^\ast$, we find that
\begin{eqnarray*}\label{xiu-2-1a}
|||(\xi_{u},\hxq)(t^\ast)|||&\leq& |||(\xi_{u},\hxq)(0)|||+2\int_0^{t^\ast}
\|\eta_{u_{t}}(s)\|ds\\
&&+C(M,\alpha,T)\int_0^{t^\ast}|||(\xi_{u},\hxq)(s)||\,ds,\nonumber
\nonumber   \hspace{-2cm}
\end{eqnarray*}
and therefore,
\begin{eqnarray*}\label{xiu-2-2a}
|||(\xi_{u},\hxq)(t)|||&\leq&|||(\xi_{u},\hxq)(0)|||+2\int_0^{t}\|\eta_{u_{t}}(s)\|\,ds\nonumber\\
&&+C(M,\alpha,T)\int_0^t|||(\xi_{u},\hxq)(s)|||\,ds.\nonumber
\nonumber
\end{eqnarray*}
An application of Gronwall lemma yields
$$
\|\xi_{u}(t)\|+\|A^\half\hxq(t)\|
\leq C\left(\|\xi_{u}(0)\| +\int_0^t||\eta_{u_{t}}||\,ds\right).
$$
Finally, a use of the triangle inequality and Theorem~\ref{thm:u-up} concludes the proof of
Theorem~\ref{FE-2}. \hfill{$\Box$}

\section{Fully discrete Scheme}
In this section, we first introduce some notations and formulate the fully discrete
DG scheme, then analyze its stability and discuss a priori error estimates.

\subsection{Notations and Scheme}
Let $\dt$ $(0<\dt<1)$ be the time step, $k=T/N$ for some positive integer $N$,  and $t_n=n\dt$.
For any function $\phi$ of time, let $\phi^n$ denote $\phi(t_n)$.
We shall use this notation for functions
defined for continuous in time as well as those defined for discrete in time.
We let $$U^\nph=\frac{U^{n+1}+U^n}{2},\qquad
U^{n;\qr}=\frac{U^{n+1}+2U^n+U^{n-1}}{4}=\frac{U^\nph+U^\nmh}{2},$$
and define the following terms for the discrete temporal derivatives:
$$
\ph U^\nph=\frac{U^{n+1}-U^n}{k},\qquad \pb U^\nph=\frac{U^\nph-U^\nmh}{k},
$$
$$
 \pc U^n=\frac{U^{n+1}-U^{n-1}}{2k}=
\frac{\ph U^\nph+\ph U^\nmh}{2},
$$
and
$$
\ps U^n=\frac{U^{n+1}-2U^n+U^{n-1}}{2k}=\frac{\ph U^\nph-\ph U^\nmh}{k}.
$$
The discrete-in-time scheme is based on a symmetric difference approximation
around the nodal points, and integral terms are computed by using
the second order quadrature formula
$$
\ii^n(\phi)=k\sum_{j=0}^{n-1}\phi(t_{j+1/2})\approx\int_0^{t_n}\phi(s)\,ds,\quad
\mbox{with} \quad t_{j+1/2}=(j+1/2)\dt.
$$
Thus, the discrete-in-time scheme for the problem
(\ref{21})-(\ref{23}) is to seek
$(U^{n},\Q^{n},\Z^{n})\in V_h\times{\bW}_h\times{\bW}_h$, such that
\begin{eqnarray}
&&\frac{2}{\dt}(\ph U^{\half},v_h)+\A_1(v_h,\Z^{\half})+
J(U^{\half},v_h)=(f^{\half}+\frac{2}{\dt}u_1,v_h), \label{7-1}\\
&&\A(\Q^{n+\half},\bw_h)-\A_1(U^{n+\half},\bw_h)+J_1(\Z^{n+\half}, \bw_h)=0,\; n\geq 0,\label{7-1a}\\
&&\A_2(\Q^{n+\half},{\boldsymbol \tau}_h)-\A(\Z^{n+\half},{\boldsymbol \tau}_h)+
\ii^{n+\half}(\B^{n+\half}(\Q,{\boldsymbol \tau}_h))=0,\; n\geq 0,\label{7-1b}\\
&&(\ps U^{n},v_h)+\A_1(v_h,\Z^{n;\qr})+
J(U^{n;\qr},v_h)=(f^{n;\qr},v_h),\; n\geq 1,\label{7-1c}
\end{eqnarray}
 for all $(v_h,{\bt}_h,\bw_h)\in V_h\times{\bW}_h\times{\bW}_h,$ with
given $(U^{0},\Q^{0},\Z^{0})\in V_h\times{\bW}_h\times{\bW}_h$. In (\ref{7-1b}),
$$
\ii ^{n+\half}(\B^{n+\half}(\Q,{\boldsymbol \tau}_h))=\frac{1}{2}\left[\ii^{n+1}
(\B^{n+1}(\Q,{\boldsymbol \tau}_h))+\ii^n(\B^n(\Q,{\boldsymbol \tau}_h)) \right],
$$
where
$$\ii^n(\B^n(\Q,\chi))=\dt\sum_{j=0}^{n-1}\B(t_n,t_{j+1/2};\Q^{j+1/2},\chi).$$
This choice of the time  discretization leads to a second order accuracy in $\dt$.

\subsection{Stability of the Discrete Problem}
We let $\Phi^n=(U^{n},\Q^{n},\Z^{n})$ and define the discrete energy norm
$$
|||\Phi^\nph|||^2=||\ph U^\nph||^2+||A^\half \Q^{n+1/2}||^2+J_1(\Z^{n+1/2},\Z^{n+1/2})+J(U^{n+1/2},U^{n+1/2}).$$
For the purpose of later error analysis, we shall first derive a stability
result for a modified scheme, namely;
\begin{eqnarray}
&& \A(\Q^{n+\half},\bw_h)-\A_1(U^{n+\half},\bw_h)+J_1(\Z^{n+\half}, \bw_h)=0,\label{7-1ak}\\
&& \A_2(\Q^{n+\half},{\boldsymbol \tau}_h)-\A(\Z^{n+\half},{\boldsymbol \tau}_h)+
\ii^{n+\half}(\B^{n+\half}(\Q,{\boldsymbol \tau}_h))=
G^\nph({\boldsymbol \tau}_h),\label{7-1bk}\\
&& (\ps U^{n},v_h)+\A_1(v_h,\Z^{n;\qr})+
J(U^{n;\qr},v_h)=(f^{n;\qr},v_h)+F^n(v_h),\label{7-1ck}
\end{eqnarray}
where $F^n$ and $G^n$ are linear functionals on $V_h$ and ${\bW}_h$,
respectively, with $G^0=0$. Set
$$
||F^n||=\sup_{\chi\in{V}_h,\chi\neq 0}
\frac{|F^n(\chi)|}{||\chi||},
$$
and similarly define the norm for $G^n$.
Then, the following stability result holds.
\begin{theorem}\label{thm:stab}
There exist positive constants $C$ and $k_0$ such that for $0<k\leq k_0$, $m\geq 1$, $t_{m+1}\leq T,$
the solution of the fully  discrete problem $(\ref{7-1ak})$-$(\ref{7-1ck})$
satisfies the following stability estimate
\begin{equation}
|||\Phi^{m+\half}|||\leq C\left\{
|||\Phi^\half|||+k\sum_{n=0}^{m+1}||f^{n}||
+k\sum_{n=1}^{m}||F^{n}||
+k\sum_{n=0}^{m}||\ph G^{n+1/2}||
\right\}.
\end{equation}
\end{theorem}

\begin{proof} We first subtract (\ref{7-1ak}) from itself with $n+1/2$ replaced
by $n-1/2$ and take the average of (\ref{7-1bk}) at two
different times. Then, we choose $\bw_h=\Z^{n;\qr}$ in
(\ref{7-1ak}),  ${\boldsymbol \tau}_h=\pc \Q^n$ in (\ref{7-1bk}),
and $v_h=\pc U^n$ in (\ref{7-1ck}) to obtain
\begin{eqnarray}
\A(\pc\Q^n,\Z^{n;\qr})&-&\A_1(\pc U^n,\Z^{n;\qr})+J_1(\pc \Z^n, \Z^{n;\qr})=0,\label{7-2aa}\\
\A_2(\Q^{n;\qr},\pc \Q^n)&-&\A(\Z^{n;\qr},\pc\Q^n)+
\frac{1}{2}\Big(\ii^{n+\half}(\B^{n+\half}(\Q,\pc\Q^n)) \label{7-2bb}\\
&+& \ii^{n-\half}(\B^{n-\half}(\Q,\pc\Q^n))\Big)=G^\nq(\pc\Q^n),\nonumber\\
(\ps U^n,\pc U^n)&+&\A_1(\pc U^n,\Z^{n;\qr})+
J(U^{n;\qr},\pc U^n) \label{7-2cc}\\
&=&(f^{n;\qr},\pc U^n)+F^n(\pc U^n). \nonumber
\end{eqnarray}
By adding, we find that
\begin{eqnarray}\label{7-3}
(\ps U^n,\pc U^n)&+&\A_2(\Q^{n;\qr},\pc \Q^n)+J_1(\pc \Z^n,
\Z^{n;\qr})+J(U^{n;\qr},\pc U^n) \nonumber\\
&=&(f^{n;\qr},\pc U^n)+F^n(\pc U^n)-\frac{1}{2}
\left(\ii^{n+\half}(\B^{n+\half}(\Q,\pc\Q^n))\right.\\
&&+\left.\ii^{n-\half}(\B^{n-\half}(\Q,\pc\Q^n))\right)+G^\nq(\pc\Q^n) \nonumber\\
&=& I_1^n+I_2^n+I_3^n+I_4^n.\nonumber
\end{eqnarray}
Notice that
$$(\ps U^n,\pc U^n)=\frac{1}{2k}\left(||\ph U^\nph||^2-||\ph U^\nmh||^2\right),$$
and
$$
\A_2(\Q^{n;\qr},\pc \Q^n)=\frac{1}{2k}\left(\A_2(\Q^{n+1/2},\Q^{n+1/2})-\A_2(\Q^{n-1/2},\Q^{n-1/2})\right).
$$
Now, we multiply both sides of (\ref{7-3}) by $2\dt$ and sum from $n=2$ to $m$,
to obtain
\begin{equation}\label{7-4}
|||\Phi^{m+1/2}|||^2
\leq |||\Phi^{3/2}|||^2
+2\dt\left|\sum_{n=2}^m(I_1^n+I_2^n+I_3^n+I_4^n)\right|.
\end{equation}
Define for some $m^{\star}$ with $ 0 \le m^{\star}\leq m,$
$$|||\Phi^{m^{\star}+1/2}||| =\max_{0\leq n\leq m}|||\Phi^\nph|||.$$
An application of the Cauchy-Schwarz inequality to (\ref{7-4}) yields
\begin{eqnarray*}
\dt\left|\sum_{n=2}^m (I_1^n+I_2^n)\right|
&\leq & \frac{\dt}{2}\sum_{n=2}^m\left(\|f^{n;\qr}\|+\|F^n\|\right)\,
\left(\|\ph U^\nph\|+\|\ph U^\nmh\| \right)\\
&\leq & \dt \sum_{n=2}^{m}\left(\|f^{n;\qr}\|+\|F^n\|\right)|||\Phi^{m^{\star}+1/2}|||.
\end{eqnarray*}
Next, we set
$$ \tilde{B}^{n+\half}_{j+\half}=\frac{1}{2}\left(B (t_{n+1},t_\jph)+ B (t_{n},t_\jph) \right).$$
In order to estimate $I_3^n$, we write
\begin{equation*}
\begin{split}
\ii^{n+\half}&(\B^{n+\half}(\Q,\pc\Q^n))\\
=&\frac{\dt}{2}\left[\sum_{j=0}^{n} \B(t_{n+1},t_{j+1/2};\Q^{j+1/2},\pc \Q^n)
+\sum_{j=0}^{n-1} \B(t_n,t_{j+1/2};\Q^{j+1/2},\pc \Q^n)\right]\\
=&\frac{1}{2}\int_\Omega B(t_{n+1},t_\nph)\Q^\nph\cdot(\Q^\nph-\Q^\nmh)\,dx\\
&+\sum_{j=0}^{n-1}\int_\Omega\Bnj \Q^\jph\cdot(\Q^\nph-\Q^\nmh)\,dx\\
=&T_1^n+T_2^n.
\end{split}
\end{equation*}
The first term $T_1$ can be bounded as follows
$$
|T_1^n|\leq\frac{M}{2} \left(\|\Q^\nph\|^2+\|\Q^\nph\|\,\|\Q^{n-\half}\|\right)
$$
showing that
\begin{equation}\label{s4}
\dt \left|\sum_{n=2}^{m}T_1^n\right|\leq
\frac{M\dt}{2\alpha}  \|A^{\half}\Q^{m+\half}\|^2 +
C(M,\alpha) \dt\left(\sum_{n=1}^{m-1}\|A^{\half}\Q^\nph\|\right)|||\Phi^{\msph}|||.
\end{equation}
For the second term $T_2^n$, we use the fact that
\begin{equation}\label{eq:formula}
H^\nph\pb\Q^\nph=\pb(H^\nph\Q^\nph)-\pb(H^\nph)\Q^\nmh,
\end{equation}
and obtain after summation
\begin{eqnarray*}
T_2^n&=&k\sum_{j=0}^{n-1}\int_\Omega \pb(\Bnj \Q^\nph)\cdot\Q^\jph\,dx-
k\sum_{j=0}^{n-1}\int_\Omega \pb\left(\Bnj\right)\Q^\jph\cdot\Q^\nmh\,dx\\
&=&k\pb\left(\sum_{j=0}^{n-1}\int_\Omega \Bnj \Q^\nph\cdot \Q^\jph\,dx\right)
+\int_\Omega{\tilde B}^{n-\half}_{n-\half}\Q^\nmh\Q^\nmh\,dx\\
&&-k\sum_{j=0}^{n-1}\int_\Omega \pb\left(\Bnj\right)\Q^\jph\cdot\Q^\nmh\,dx.
\end{eqnarray*}
Hence, we arrive at
\begin{eqnarray*}
k\left|\sum_{n=2}^{m}T_2^n\right|&\leq&
k^2\left|\sum_{j=0}^{m-1}\int_\Omega {\tilde B}^{m+\half}_\jph \Q^\jph\cdot \Q^{m+\half}\,dx
-\int_\Omega {\tilde B}^{3/2}_{1/2}\Q^{3/2}\cdot\Q^{1/2}\,dx\right|\\
&&+k\left|\sum_{n=2}^{m}\int_\Omega{\tilde B}^{n-\half}_{n-\half}\Q^\nmh\cdot\Q^\nmh\,dx\right|\\
&&+k^2\left|\sum_{n=2}^{m}\sum_{j=0}^{n-1}\int_\Omega \pb\left(\Bnj\right)\Q^\jph\cdot\Q^\nmh\,dx\right|
\\
&\leq& M k^2 \|\Q^{m+\half}\|\sum_{j=0}^{m-1}\|\Q^\jph\|+Mk^2\|\Q^{3/2}\|\,\|\Q^{1/2}\|\\
&&+Mk\sum_{n=2}^{m}\|\Q^\nmh\|^2
+MTk\sum_{j=0}^{m-1} \|\Q^\jph\|^2\\
&\leq& C(M,\alpha,T)k\left(\sum_{n=0}^{m-1}\|A^{\half}\Q^\nph\|\right)|||\Phi^{\msph}|||.
\end{eqnarray*}
We can now estimate $\ii^{n-\half}(\B^{n-\half}(\Q,\pc\Q^n))$ in a similar way, but without
having the term $\|A^{\half}\Q^{m+\half}\|^2$ on the right hand
side of (\ref{s4}), and thus, obtain
$$
k\left|\sum_{n=2}^{m}I_3^n\right|\leq
\frac{Mk}{4\alpha}\|A^{\half}\Q^{m+\half}\|^2
+ C(M,\alpha,T)k\left(\sum_{n=0}^{m-1}\|A^{\half}\Q^\nph\|\right)|||\Phi^{\msph}|||.
$$
For the last term  $I_4=kG^\nq(\pc\Q^n)$,
we  again use summation by parts technique to arrive at
$$
k\left|\sum_{n=2}^{m}I_4^n\right|
\leq\left|G^{m;1/4}(\Q^{m+1/2})-G^{1;1/4}(\Q^{3/2})-
k\sum_{n=2}^{m}\pb (G^{n;1/4})(\Q^\nmh)\right|.$$
Notice that, since $G^0=0$, it follows that
$$
G^n=k\sum_{n=0}^{n-1}\ph G^\nph,
$$
and hence, we obtain
$$
k\left|\sum_{n=2}^{m}I_4^n\right|\leq C(\alpha)k
\left(\sum_{n=0}^{m}\|\ph G^{n+\half}\|\right)|||\Phi^{\msph}|||.
$$
It remains now to bound the  term $|||\Phi^{3/2}|||$ on the right hand side of (\ref{7-4}).
Equation (\ref{7-3}) taken at $n=1$ yields
\begin{eqnarray}\label{7-3b}
|||\Phi^{3/2}|||^2&\leq&|||\Phi^{1/2}|||^2+2k\left|(f^{1;1/4},\pc U^1)+
F^{1}(\pc U^1)-\frac{1}{2}\left(\ii^{3/2}(\B^{3/2}(\Q,\pc\Q^1))\right.\right.\nonumber\\
&&\left.\left.+\ii^{1/2}(\B^{1/2}(\Q,\pc\Q^1)\right)+G^{1;1/4}(\pc\Q^1)\right|\nonumber\\
&\leq&|||\Phi^{1/2}|||^2+C(M,\alpha)k\left(||f^{1;1/4}\|+\|F^1\|+
||A^\half\Q^{1/2}||\right.\\
&&\left.+\|A^\half\Q^{3/2}\|+\|\ph G^{1/2}\|+\|\ph G^{3/2}\|\right)
|||\Phi^{\msph}|||.\nonumber
\end{eqnarray}
Now, substitute  estimates involving $I_1^n,\cdots,I_4^n$ and (\ref{7-3b})  in (\ref{7-4}) to find that
\begin{eqnarray}\label{7-3c}
\Big(1-\frac{M}{2\alpha}k\Big)\,|||\Phi^{m+\half}|||^2&\leq & |||\Phi^{1/2}|||^2
+Ck\Big\{\sum_{n=0}^{m+1}\|f^n\|+\sum_{n=1}^{m}\|F^n\| \nonumber\\
&& +\sum_{n=0}^{m}\|\ph G^{n+\half}\|+\sum_{n=0}^{m-1}|||\Phi^\nph|||\Big\}
|||\Phi^{\msph}|||,
\end{eqnarray}
Choose $k_0>0$ such that for $0 < k \leq k_0,$ $ (1-\frac{M}{2\alpha}k )>0.$ Then
replace $m$ by $m^{\star}$ in (\ref{7-3c}) to obtain
\begin{eqnarray}\label{7-3d}
|||\Phi^{\msph}||| &\leq &  C\Big\{ |||\Phi^{1/2}|||
+ k \sum_{n=0}^{m^{\star}+1}\|f^n\|+k\sum_{n=1}^{m^{\star}}\|F^n\| \\
&+& k\sum_{n=0}^{m^{\star}}\|\ph G^{n+\half}\|+k\sum_{n=0}^{m^{\star}-1}|||\Phi^\nph|||\Big\},\nonumber
\end{eqnarray}
and hence, replacing $m^{\star}$ by $m$ on the right hand side of (\ref{7-3d}), it follows that
\begin{eqnarray}\label{7-3e}
|||\Phi^{m+\half}||| &\leq& |||\Phi^{\msph}||| \leq   C\Big\{ |||\Phi^{1/2}|||
+ k \sum_{n=0}^{m+1}\|f^n\|+k\sum_{n=1}^{m}\|F^n\| \nonumber\\
&+& k\sum_{n=0}^{m}\|\ph G^{n+\half}\|+k\sum_{n=0}^{m-1}|||\Phi^\nph|||\Big\}.
\end{eqnarray}
An application of the discrete Gronwall lemma to (\ref{7-3e})
completes the rest of the proof.
\end{proof}

\subsection{Convergence Analysis}
For ${\boldsymbol\phi}\in\bW_h$, we define a linear functional
$\ce^{n}({\boldsymbol\phi})$  representing the error in
 the quadrature formula by
$$
\ce^{n}({\boldsymbol\phi})({\boldsymbol\chi})=\ii^{n}\left(\B^{n}({\boldsymbol\phi},
{\boldsymbol\chi})\right)-\int_0^{t_n}\B(t_n,s;{\boldsymbol\phi},{\boldsymbol\chi})\,ds.
$$
Notice that $\ce^{0}({\boldsymbol\phi})=0$.
In our analysis, we shall use the following lemma which can be found in
\cite{PTW}.
\begin{lemma}\label{lm:8-1}
There exists a positive constant $C,$ independent of $h$ and $k$ such that
the following estimates
$$
k\sum_{n=0}^{m}||\ce^{n+1}(\tilde{\bq}_h)||\leq Ck^2
\int_0^{t_{m+1}}(||\tilde{\bq}_h||+||\tilde{\bq}_{ht}||+
||\tilde{\bq}_{htt}||)\,ds,
$$
and
$$
k\sum_{n=0}^{m}||\ph\ce^{n+1/2}(\tilde{\bq}_h)||\leq Ck^2
\int_0^{t_{m+1}}(||\tilde{\bq}_h||+||\tilde{\bq}_{ht}||+
||\tilde{\bq}_{htt}||)\,ds,
$$
hold.
\end{lemma}
In order to derive {\it a priori} error estimates for the fully discrete scheme, we rewrite
\begin{eqnarray*}
u^n-U^n:=(u^n-\tilde u^n_h)-(U^n-\tilde u^n_h)=:\eU^n-\xU^n,\\
          {\bq^n-\Q^n}:=(\bq^n-\tilde{\bq}^n_h)-(\Q^n-\tilde{\bq}^n_h)=:\eQ^n-\xQ^n,\\
          \bs^n-\Z^n:=(\bs^n-\tilde{\bs}^n_h)-(\Z^n-\tilde{\bs}^n_h)=:\eZ^n-\xZ^n.
\end{eqnarray*}
The following theorem provides {\it a priori} error estimates for the fully discrete scheme.
\begin{theorem}\label{thm:6-1}
Let $(u,\bq,\bs)$ be the solution of $(\ref{21})$-$(\ref{23})$.
Further, let $(U^n,\Q^n,\Z^n)\in V_h\times{\bW}_h\times  {\bW}_h$ be the
solution of $(\ref{7-1})$-$(\ref{7-1c})$. Assume that $U^0=\tilde{u}_0$,
$\Q^0={\bf I}_h\nabla u_0$ and $\Z^0={\bf I}_h(A\nabla u_0)$.
Then there exists  constants $C>0,$ independent of $h$ and $k$,
and  $k_0>0,$ such that for $0<k<k_0$ and
$m=0,1,\cdots,N-1$
\begin{eqnarray}\label{eu2-d}
&&\|\ph (u(t_{m+\half})-U^{m+\half})\|_{L^2(\Omega)}
\leq  C \frac{h^{P+D}}{p^{\sQ+\sZ}}\Big(\|u_0\|_{H^{r+2}(\mathcal{T}_h)} +\|u_1\|_{H^{r+2}(\mathcal{T}_h)} \\
&& \quad +\sum_{j=0}^{2}\left\|\frac{\partial^{j}u }{\partial t^j}\right\|_{L^1(H^{r+2}(\mathcal{T}_h))}\Big)
+ Ck^2\sum_{j=3}^{4}\int_{0}^{t_{m+1}}
\left\|\frac{\partial^j u}{{\partial t^j}}(s)\right\|\,ds,\nonumber
\end{eqnarray}
and
\begin{eqnarray}\label{eu2-d-b}
&&\|{\bf q}(t_{m+\half})-\Q^{m+\half}\|_{L^2(\Omega)}+
\|\bs(t_{m+\half})-\Z^{m+\half}\|_{L^2(\Omega)}
\leq \\
&& \qquad C \frac{h^{P}}{p^{\sQ}}\Big(\|u_0\|_{H^{r+2}(\Omega_h)}
 +\|u_1\|_{H^{r+2}(\Omega_h)}+\sum_{j=0}^{2}\left\|\frac{\partial^{j}u }
{\partial t^j}\right\|_{L^1(H^{r+2}(\Omega_h))}\Big)\nonumber\\
&& \qquad + Ck^2\sum_{j=3}^{4}\int_{0}^{t_{m+1}}
\left\|\frac{\partial^j u}{{\partial t^j}}(s)\right\|ds,\nonumber
\end{eqnarray}
where $P=\min\{r+\frac{1}{2}(1+\mu_*),p+\frac{1}{2}(1-\mu^*)\},~D=\frac{1}{2}(1+\mu_*)$,
$\sQ=r+\min\{\mu_*,1-\mu^*\}$ and $\sZ=\min\{\frac{1}{2},\mu_*\}$.
\end{theorem}
\begin{proof}
Since estimates for $\eU$, $\eQ$ and $\eZ$ are known from Theorem~\ref{thm:u-up}, it is enough to estimate $\xU$, $\xQ$ and
$\xZ$. Using (\ref{p1})-(\ref{p3}), we derive the following system
\begin{eqnarray}
&&\frac{2}{\dt}(\ph \xU^{\half},v_h)+\A_1(v_h,\xZ^{\half})+
J(\xU^{\half},v_h)=\frac{2}{\dt}(\ph \eU^{\half},v_h)+(2r^0,v_h)\label{8-0}\\
&& \A(\xQ^\nph, \bw_h)-\A_1(\xU^\nph,\bw_h)+J_1(\xZ^\nph,\bw_h)=0, \label{8-1}\\
&& \A_2(\xQ^\nph,{\bt}_h)-\A(\xZ^\nph, {\bt}_h )+
  \ii^{n+\half}(\B^{n+\half}(\xQ,{\bt}_h))\label{8-2}\\
&&\hspace*{6cm}=-\ce^{n+1/2}(\tilde{\bq}_h)({\bt}_h)\nonumber\\
&&(\ps \xU^n, v_h)+ \A_1(v_h,\xZ^{n;1/4})+J(\xU^{n;1/4},v_h)=
(\ps\eU^n, v_h)+(r^n,v_h),\label{8-3}
\hspace{-2cm}
\end{eqnarray}
where $\displaystyle r^0=\frac{1}{2}u_{tt}^\half+\frac{1}{\dt}
\left(u_1-\ph u^\half\right),$ and
$$
r^n=u_{tt}^{n;1/4}-\ps u^n=\frac{1}{12}\\
\int_{-k}^{k}(|t|-k)\left(3-2(1-|t|/k)^2\right)\frac{\partial^4 u}{\partial t^4}
(t^n+t)dt,\;\;n\geq 1.
$$
With $F^n(v_h)=(\ps\eU^n, v_h)+(r^n,v_h)$ and
$G^n({\bt}_h)=\ce^{n}(\tilde{\bq}_h)({\bt}_h)$,
we apply Theorem~\ref{thm:stab} to arrive at
$$
|||\Psi^{m+\half}|||\leq C\left\{
|||\Psi^\half|||+k\sum_{n=1}^{m}\big(||\ps\eU^n||+||r^n||\big)
+k\sum_{n=0}^{m}||\ph \ce^{n+1/2}(\tilde{\bq}_h)||
\right\},
$$
where $\Psi^n=(\xU^n,\xQ^n,\xZ^n)$. To estimate $\Psi^\half$ on the right hand side
of this inequality, we first note that,
$\xU^0=0$ since $U^0=\tilde{u}_0$, and hence,
$\displaystyle\xU^\half=\frac{k}{2}\ph \xU^\half$.
Now,  choose $v_h=\xU^\half$
in (\ref{8-0}), $v_h=\xZ^\half$ in (\ref{8-1}), and $v_h=\xQ^\half$ in (\ref{8-2}).
Adding the resulting equations, and taking into account that $\ii^0(\phi)=0$
and $\ce^0(\phi)=0$, we obtain
\begin{eqnarray*}
|||\Psi^\half|||^2
&\leq&
\left|\left(\ph\eU^\half,\ph \xU^\half\right)\right|+k\left|\left(r^0,\ph \xU^\half\right)\right|
+\frac{1}{2}\left|\ii^1(\B^1(\xQ,\xQ^\half)\right|
+\frac{1}{2}\left|\ce^1(\tilde{\bq}_h)(\xQ^\half)\right|\\
&\leq& C(\alpha)
\left(||\ph\eU^\half||+k||r^0||+||\ce^1(\tilde{\bq}_h)||\right)|||\Psi^\half||| +
\frac{Mk}{2\alpha}||A^\half\xQ^\half||^2 \\
&\leq& C(\alpha)
\left(||\ph\eU^\half||+k||r^0||+||\ce^1(\tilde{\bq}_h)||\right)|||\Psi^\half||| +
\frac{Mk}{2\alpha}|||\Psi^\half|||^2.
\end{eqnarray*}
For $ 0<k\leq k_0,$ $\big(1-(Mk)/(2\alpha)\big) >0 $ and hence,
$$
|||\Psi^\half|||\leq C\left\{||\ph\eU^\half||+k||r^0||+||\ce^1(\tilde{\bq}_h)||\right\},
$$
which shows that
\begin{equation}\label{8-5}
|||\Psi^{m+\half}|||\leq C\left\{
||\ph\eU^\half||+k\sum_{n=1}^{m}||\ps\eU^n||+k\sum_{n=0}^{m}||r^n||
+k\sum_{n=0}^{m}||\ph \ce^{n+1/2}(\tilde{\bq}_h)||
\right\}.
\end{equation}
To estimate the terms on the right hand side of (\ref{8-5}), we note that
\begin{eqnarray}\label{8-5a}
||\ph\eU^\half||\leq\frac{1}{k}\int_0^{\dt}||\eta_{\U t}(s)||\,ds,
\end{eqnarray}
and by Taylor series expansions, we arrive at
\begin{eqnarray}\label{8-5b}
k\sum_{n=1}^{m}||\ps\eU^n||&\leq&\frac{1}{k}\sum_{n=1}^{m}\left\{
\int_{t_n}^{t_{n+1}}(t_{n+1}-s)||\eta_{\U tt}(s)||\,ds+
\int_{t_{n-1}}^{t_n}(s-t_{n-1})||\eta_{\U tt}(s)||\,ds
\right\}\nonumber\\
&\leq&2\int_{0}^{t_{m+1}}||\eta_{\U tt}(s)||\,ds.
\end{eqnarray}
Further, we find that
$$||r^n||\leq Ck\int_{t_{n-1}}^{t_{n+1}}\left|\left|
\frac{\partial^4 u}{{\partial t^4}}(s)\right|\right|\,ds,\quad n\geq 1,$$
and
$$||r^0||\leq Ck||u_{ttt}||_{L^\infty(0,\dt/2;L^2(\Omega))}
\leq Ck\int_{0}^{t_{m+1}}
\left(\left|\left|\frac{\partial^3 u}{{\partial t^3}}(s)\right|\right|
+\left|\left|\frac{\partial^4 u}{{\partial t^4}}(s)\right|\right|\right)\,ds.$$
Hence,
\begin{equation}\label{s2}
k\sum_{n=0}^{m}||r^n||\leq Ck^2\int_{0}^{t_{m+1}}
\left(\left|\left|\frac{\partial^3 u}{{\partial t^3}}(s)\right|\right|
+\left|\left|\frac{\partial^4 u}{{\partial t^4}}(s)\right|\right|\right)\,ds.
\end{equation}
For the last term in (\ref{8-5}), a use of Lemma~\ref{lm:8-1} with the triangle
inequality yields
$$
k\sum_{n=0}^{m}||\ph\ce^{n+1/2}(\tilde{\bq}_h)||\leq Ck^2
\sum_{j=0}^2\int_{0}^{t_{m+1}}
\left(\left|\left|\frac{\partial^j \bq}{{\partial t^j}}(s)\right|\right|
+\left|\left|\frac{\partial^j \eta_{\bq}}{{\partial t^j}}(s)\right|\right|\right)\,ds,
$$
and hence, using the estimates  in Theorem~\ref{thm:u-up}, we deduce that
\begin{equation}\label{s3}
k\sum_{n=0}^{m}||\ph\ce^{n+1/2}(\tilde{\bq}_h)||\leq Ck^2
\sum_{j=0}^2\int_0^{t_{m+1}}
\left|\left|\frac{\partial^j u}{{\partial
t^j}}(s)\right|\right|_{H^{r+2}(\Omega)}ds.
\end{equation}
Substitute (\ref{8-5a})-(\ref{s3}) in (\ref{8-5}) and use the triangle inequality with  Theorem~\ref{thm:u-up}
we obtain the error estimates involving $u$ and $\bq$ in (\ref{eu2-d})-(\ref{eu2-d-b}). Now to complete the proof of  (\ref{eu2-d-b}), it remains to estimate $||\xZ^{m+\half}||$.
To do so, choose
${\bt}_h=\xZ^{m+\half}$ in (\ref{8-2}), and conclude that
$$||\xZ^{m+\half}||\leq C(M,T)(\max_{0\leq n\leq m}||\xQ^{n+\half}|| + ||\ce^{m+1/2}(\tilde{\bq}_h)||).$$
Finally, a use of the triangle inequality and Theorem~\ref{thm:u-up}
completes the rest of the proof.
\end{proof}

Below, we again recall a variant of Baker's nonstandard energy formulation
to prove  $L^2$-estimate for the error under reduced regularity conditions.
We shall introduce the following notations for proving the final theorem of  this paper.
Define
$$
\hat{\phi}^0=0,\qquad \hat{\phi}^n=k\sum_{j=0}^{n-1}\phi^{\jph}.
$$
Then,
$$\ph \hat{\phi}^\nph={\phi}^\nph,$$ and
$$k\sum_{j=1}^{n}\phi^{j;1/4}=\hat{\phi}^{\nph}-\frac{k}{2}\phi^{\half}.$$
Let $R^n=k\sum_{j=0}^{n}r^n$.
Multiplying (\ref{8-2}) and (\ref{8-3}) by $k$ and summing over $n$, we
derive the new system
\begin{eqnarray}
&& \A(\xQ^\nph, \bw_h)-\A_1(\xU^\nph,\bw_h)+J_1(\xZ^\nph,\bw_h)=0, \label{8-1b}\\
&& \A_2(\xQh^\nph,{\bt}_h)-\A(\xZh^\nph, {\bt}_h)+
\hii^{n+\half}(\B(\xQ,{\bt}_h))\label{8-2b}\\
&&\hspace*{6cm}=-\hce^{n+1/2}(\tilde{\bq}_h)({\bt}_h)\nonumber\\
&&(\ph\xU^\nph, v_h)+ \A_1(v_h,\xZh^\nph)+J(\xUh^\nph,v_h)=
(\ph\eU^\nph, v_h)+(R^n,v_h),\label{8-3b}
\hspace{-2cm}
\end{eqnarray}
where the meanings of $\hii^{n+\half}(\B(\xQ,{\bt}_h))$ and $\hce^{n+1/2}(\tilde{\bq}_h)({\bt}_h)$ are obvious. Notice that (\ref{8-3b}) is derived after cancellation using (\ref{8-0}).
Also, remark that (\ref{8-3b}) reduces to (\ref{8-0}) when $n=0$.
Now, choose $v_h=\xZh^\nph$
in (\ref{8-1b}), ${\bt}_h=\xQ^\nph$ in (\ref{8-2b}), and
$v_h=\xU^\nph$ in (\ref{8-3b}).
Adding the resulting equations, we find that
\begin{eqnarray} \label{8-4b}
(\ph\xU^\nph,\xU^\nph)&+&\A_2(\xQh^\nph,\xQ^\nph)+J_1(\xZ^\nph,\xZh^\nph)+
J(\xUh^\nph,\xU^\nph)\nonumber\\
&=&-\hii^{n+\half}(\B(\xQ,\xQ^\nph))
-\hce^{n+1/2}(\tilde{\bq}_h)(\xQ^\nph) \nonumber\\
&&+ (\ph\eU^\nph,\xU^\nph)+(R^n,\xU^\nph).
\end{eqnarray}
Note that $\xi_U^0=0$,
$\xQh^0=0$ and $\xQ^\nph=\ph \xQh^\nph,$ and
$$
(\ph\xU^\nph, \xU^\nph)=\frac{1}{2k}
\left(||\xU^{n+1}||^2-||\xU^{n}||^2\right).
$$
Other terms on the left hand side of (\ref{8-4b}) can be rewritten in a similar way.
On substitution and then
summing from $n=0$ to $m$ after  multiplying by $2k,$
we arrive at
\begin{eqnarray}\label{8-5b}
||\xU^{m+1}||^2&+&\A_2(\xQh^{m+1},\xQh^{m+1})+J_1(\xZh^{m+1},\xZh^{m+1})+
J(\xUh^{m+1},\xUh^{m+1})\nonumber\\
&=&-2k\sum_{n=0}^{m}\hii^{n+\half}(\B(\xQ,\xQ^\nph))
-2k\sum_{n=0}^{m}\hce^{n+1/2}(\tilde{\bq}_h)(\xQ^\nph)\nonumber\\
&+&2k\sum_{n=0}^{m}(\ph\eU^\nph,\xU^\nph)+2k\sum_{n=0}^{m}(R^n,\xU^\nph)\\
&=&2(I_1^m+I_2^m+I_3^m+I_4^m).\nonumber
\end{eqnarray}
For convenience, we use the following notations: $B_r^s=B(t_{s},t_r)$, and
$\Delta_i\phi_i=(\phi_i-\phi_{i-1})/k$. Further, let
$|||(\xU^n,\xQh^n)|||^2=||\xU^n||^2+||A^\half\xQh^n||^2$ and  for some $m^{\star}\in [0;m+1]$, define
 $$|||(\xU^{m^{\star}},\xQh^{m^{\star}})|||=\max_{0\leq n\leq m+1}|||(\xU^n,\xQh^n)|||.$$
For the $I_1^m$ term, we observe that
\begin{eqnarray*}
\ii^{j+1}(B^{j+1}(\xQ,\xQ^\nph))
&=& k\sum_{i=0}^jB^{j+1}_\iph\xQ^\iph\cdot\xQ^\nph
=k\sum_{i=0}^jB^{j+1}_\iph\ph\xQh^\iph\cdot\xQ^\nph\\
&=& \sum_{i=0}^j\left(B^{j+1}_\iph\xQh^{i+1}-B^{j+1}_\imh\xQh^i\right)\cdot\xQ^\nph\\
&&-\sum_{i=0}^j\left(B^{j+1}_\iph-B^{j+1}_\imh\right)\xQh^i\cdot\xQ^\nph\\
&=&B^{j+1}_\jph\xQh^\jph\cdot\ph\xQh^{n+1}-k\sum_{i=0}^j\left(\Delta_i(B^{j+1}_\iph)\right)
\xQh^i\cdot\ph\xQh^\nph.
\end{eqnarray*}
Hence,
$$
k^2 \sum_{n=0}^m\sum_{j=0}^n\ii^{j+1}(B^{j+1}(\xQ,\xQ^\nph))=
k^2\sum_{n=0}^m\Theta^{n+1}\cdot\xQ^\nph-k^2\sum_{n=0}^m\Upsilon^{n+1}\cdot\xQ^\nph,
$$
where
$$\Theta^{n+1}=\sum_{j=0}^nB^{j+1}_\jph\xQh^{j+1}\quad\mbox{ and } \quad
\Upsilon^{n+1}=k\sum_{j=0}^n\sum_{i=0}^j\left(\Delta_i(B^{j+1}_\iph)\right)\xQh^i.
$$
Next, we estimate the terms $\Theta^{n+1}$ and $\Upsilon^{n+1}$. Note that
\begin{eqnarray*}
k^2\sum_{n=0}^m\Theta^{n+1}\cdot\ph\xQh^\nph
&=&k\sum_{n=0}^m(\Theta^{n+1}\cdot\xQh^{n+1}-\Theta^{n}\cdot\xQh^{n})
-k\sum_{n=0}^m(\Theta^{n+1}-\Theta^{n})\cdot\xQh^{n}\\
&=&k\Theta^{m+1}\cdot\xQh^{m+1}-k\sum_{n=0}^mB^{n+1}_\nph\xQh^{n+1}\cdot\xQh^{n}\\
&=&k\sum_{n=0}^m B^{n+1}_\nph \xQh^{n+1}\cdot\xQh^{m+1}-k\sum_{n=0}^mB^{n+1}_\nph
\xQh^{n+1}\cdot\xQh^{n}.
\end{eqnarray*}
Therefore, using the Cauchy-Schwarz inequality,  we arrive at
$$
\left|k^2\sum_{n=0}^m\Theta^{n+1}\cdot\ph\xQh^\nph\right|
\leq
\frac{Mk}{\alpha} ||A^\half\xQh^{m+1}||^2+C(M,\alpha)\left(k\sum_{n=0}^m ||A^\half\xQh^{n}||\right)
|||(\xU^{m^{\star}},\xQh^{m^{\star}})|||.
$$
Similarly, we now obtain
\begin{eqnarray*}
k^2\sum_{n=0}^m\Upsilon^{n+1}\cdot\ph\xQh^\nph
&=&k\Upsilon^{m+1}\cdot\xQh^{m+1}-k\sum_{n=0}^m(\Upsilon^{n+1}-\Upsilon^{n})\cdot\xQh^{n}\\
&=&k^2\left(\sum_{j=0}^m\sum_{i=0}^j\left(\Delta_i(B^{j+1}_\iph)\right)
\xQh^i\right)\cdot\xQh^{m+1}\\
&&-k^2\sum_{n=0}^m\left(\sum_{i=0}^n\left(\Delta_i(B^{n+1}_\iph)\right)\xQh^i
 \right)\cdot\xQh^{n}.
\end{eqnarray*}
Since, it is assumed that  $||D_{t,s}B(t,s)||\leq M$, we deduce that
$$
\left|k^2\sum_{n=0}^m\Upsilon^{n+1}\cdot\ph\xQh^\nph\right|
\leq C(M,\alpha,T)\left(k\sum_{n=0}^m ||A^\half\xQh^{n}||\right)|||(\xU^{m^{\star}},\xQh^{m^{\star}})|||.
$$
Since, similar bounds can be obtained for the other terms in $I_1^m$,
we finally conclude that
$$
|I_1^m|\leq \frac{Mk}{2\alpha} ||A^\half\xQh^{m+1}||^2+
C(M,\alpha,T)\left(k\sum_{n=0}^m ||A^\half\xQh^{n}||\right)|||(\xU^{m^{\star}},\xQh^{m^{\star}})|||.
$$
Now, with $\Lambda^{n+1}=\sum_{j=0}^{n}\ce^{j+1}(\tilde{\bq}_h)$, we observe that
\begin{eqnarray*}
k^2\sum_{n=0}^{m}\sum_{j=0}^{n}\ce^{j+1}(\tilde{\bq}_h)(\xQ^\nph)
&=&k^2\sum_{n=0}^{m}\Lambda^{n+1}\ph\xQh^\nph\\
&=&k\Lambda^{m+1}(\xQh^{m+1})-k\sum_{n=0}^m(\Lambda^{n+1}-\Lambda^{n})(\xQh^{n})\\
&=&k\sum_{j=0}^{m}\ce^{j+1}(\tilde{\bq}_h)(\xQh^{m+1})
-k\sum_{n=0}^{m}\ce^{n+1}(\tilde{\bq}_h)(\xQh^n).
\end{eqnarray*}
Since, the terms in $I_2^m$ have a similar form, we deduce that
$$
|I_2^m|\leq C(\alpha)k\sum_{n=0}^{m}\left(||\ce^{n+1}(\tilde{\bq}_h)||\right)
|||(\xU^{m^{\star}},\xQh^{m^{\star}})|||.
$$
Finally, it follows that
$$
|I_3^m+I_4^m|\leq Ck \sum_{n=0}^{m}\left(\left|\left|\ph\eU^\nph\right|\right|+||R^n||\right)
|||(\xU^{m^{\star}},\xQh^{m^{\star}})|||.
$$
On substituting the above estimates in (\ref{8-5b}), using kickback arguments and similar arguments  in previous Theorems, we arrive at
\begin{eqnarray*}
\Big(1-(M/2\alpha)k\Big)\,|||(\xU^{m+1},\xQh^{m+1})||| &\leq &C
k \sum_{n=0}^{m}\left(
\left|\left|\ph\eU^\nph\right|\right|+||R^n||+||\ce^{n+1}(\tilde{\bq}_h)||+||A^\half\xQh^{n}||\right)\\
 &\leq & Ck \sum_{n=0}^{m}\left(
\left|\left|\ph\eU^\nph\right|\right|+||R^n||+||\ce^{n+1}(\tilde{\bq}_h)||+|||(\xU^{n},\xQh^{n})|||\right).
\end{eqnarray*}
Since for $0 < k \leq k_0,$ $\big(1-(M/2\alpha)k\big)$ can be made positive, then
an application of the  discrete Gronwall lemma yields
$$||\xU^{m+1}||+||A^\half\xQh^{m+1}||\leq
Ck \sum_{n=0}^{m}\left(
\left|\left|\ph\eU^\nph\right|\right|+||R^n||+||\ce^{n+1}(\tilde{\bq}_h)||\right).
$$
The first two terms on the right hand side can be bounded as follows:
$$
k\sum_{n=0}^{m}\left|\left|\ph\eU^\nph\right|\right|\leq
\int_0^{t_{m+1}}\left|\left|\eta_{\U t}(s)\right|\right|\,ds,
$$
and
$$
k\sum_{n=0}^{m}||R^n||\leq T\max_{0\leq n\leq m}||R^n||\leq
kT\sum_{n=0}^{m}||r^n||.
$$
Finally, by taking into account (\ref{s2}) and (\ref{s3}), we  use the
triangle inequality with the first estimate in Theorem~\ref{thm:u-up}
to prove the error estimates in the following theorem.
\begin{theorem}\label{thm:6-2}
Let $(u,\bq,\bs)$ be the solution of $(\ref{21})$-$(\ref{23})$.
Further, let $(U^n,\Q^n,\Z^n)\in V_h\times{\bW}_h\times  {\bW}_h$ be the
solution of $(\ref{7-1})$-$(\ref{7-1c})$. Assume that $U^0=\tilde{u}_0$,
$\Q^0={\bf I}_h\nabla u_0$ and $\Z^0={\bf I}_h(A\nabla u_0)$.
Then there exists a constant $C>0$ independent of $h$ and $k$,
and a constant $k_0>0$ independent of $h$, such that for $0<k<k_0$ and
$m=0,1,\cdots,N-1$
\begin{eqnarray*}
\qquad \|u(t_{m+1})-U^{m+1}\|_{L^2(\Omega)}
&\leq&  C \frac{h^{P+D}}{p^{\sQ+\sZ}}\Big(\|u_0\|_{H^{r+2}(\mathcal{T}_h)}
+\sum_{j=0}^{1}\left\|\frac{\partial^{j}u }{\partial t^j}\right\|_{L^1(H^{r+2}(\mathcal{T}_h))}\Big)\nonumber\\
&&+ Ck^2\sum_{j=3}^{4}\int_{0}^{t_{m+1}}
\left|\left|\frac{\partial^j u}{{\partial t^j}}(s)\right|\right|\,ds,
\end{eqnarray*}
where $P,D,\sQ$ and $\sZ$ as in Theorem~$\ref{thm:6-1}.$
\end{theorem}

\begin{remark}
One of the difficulties involved in the time stepping scheme applied to the present problem on hyperbolic integro-differential equation is that all the values of  $\Q^j,\,j=0,\cdots,\cdots, n-1$ have to be retained in order to compute the current value, say $\Q^n.$ This, in turn, causes a  great demand for data storage. One way to overcome this difficulty is to use sparse quadrature rules proposed in Sloan and Thom\'ee \cite{st} and analyzed for hyperbolic problems by Pani {\it et al.} \cite{PTW}. This way, one may substantially reduce the storage requirements in the computation. Since the error analysis for the present problem with sparse quadrature rules will be quite involved, it is not possible to address it in this article.
\end{remark}

 \section{Numerical Experiments}
In this section, we present the performance of the proposed DG methods for
the linear hyperbolic integro-differential equations of the form (\ref{1})-(\ref{4}) with $A=I,$ $B(x,t,s)=\exp(t-s)I$,
$\Omega=(0,1)\times(0,1)$ and $T=1$.
We divide $\Omega$ into regular uniform closed triangles and let $0=t_0<t_1<\cdots < t_N=T$
be a given partition of the time interval $(0,T]$ with step length $\dt=\frac{T}{N}$ for
some positive integer $N$.

We first  discuss the numerical procedure for the DG schemes. Let
$(\phi_i)_{i=1}^{N_h}$ be the basis functions for the finite dimensional
space $V_h$, where $N_h$ denotes the dimension of the space $V_h$ and let $(\boldsymbol\chi_i)_{i=1}^{M_h}$ be the basis functions for the finite dimensional
space ${\bW}_h$, where $M_h$ denotes the dimension of the space ${\bW}_h$.

Then, we define the following matrices
\begin{eqnarray*}
\label{nu2}
M&=&[M(i,j)]_{1\leq i ,j\leq N_h},~~~A_1=[A_1({i,j})]_{1\leq i\leq N_h,1\leq j\leq M_h },\\
J_1&=&[J_1(i,j)]_{1\leq i, j\leq M_h },~~~
A=[A({i,j})]_{1\leq i,j\leq M_h},~~~J=[J({i,j})]_{1\leq i ,j\leq N_h},\;\;\;\;\;
\end{eqnarray*}
and the vector
\begin{eqnarray}
\label{nu3}
L=[L(i)]_{1\leq i\leq N_h},
\end{eqnarray}
where
\begin{eqnarray*}
M(i,j)&=&\sum_K \int_K \phi_i\phi_j \,dx,\hspace{1cm}A_1({i,j})=\sum_K \int_K \phi_i\cdot \nabla\boldsymbol \chi_j \,dx -\int_\Gamma\{\!\!\{\boldsymbol\chi_j\}\!\!\}[\![\phi_i]\!]dS,\hspace{-1cm}\\
J_1(i,j)&=& \sum_{e\in \Gamma_I}\int_e C_{22}[\![ \chi_i]\!][\![ \chi_j]\!]\,dS,\hspace{1cm}A({i,j})(t)= \int_\Omega{\boldsymbol\chi}_i\cdot {\boldsymbol \chi}_j \,dx,\\
J({i,j})&=&\sum_{e\in \Gamma}\int_e C_{11}[\![ \phi_i]\!][\![ \phi_j]\!] \,dS,
\end{eqnarray*}
and $~~L(i)=\displaystyle\sum_K\int_{K}f \phi_i \,dx.$

\noindent
Write $U^n=\displaystyle\sum_{i=1}^{N_h}\alpha_i^n\phi_i $, where
${\boldsymbol\alpha}^n=[\alpha_1^n,\alpha_2^n,\cdots,\alpha_{N_h}^n]$,
$\Q^n=\displaystyle\sum_{i=1}^{M_h}\beta_i^n\boldsymbol\chi_i $, where
$ {\boldsymbol\beta}^n=[\beta_1^n,\beta_2^n,\cdots\beta_{M_h}^n]$ and
$\Z^n=\displaystyle\sum_{i=1}^{M_h}\gamma_i^n\boldsymbol\chi_i $, where
$ {\boldsymbol\gamma}^n=[\gamma_1^n,\gamma_2^n,\cdots,\gamma_{M_h}^n]$ \\

Now using the basis functions for $V_h$ and ${\bW}_h$,
(\ref{7-1a})-(\ref{7-1c}) can be reduced to the following matrix form:
\begin{eqnarray}
A\boldsymbol\beta^{n+\half}+A_1\boldsymbol\alpha^{n+\half}+J_1\boldsymbol\gamma^{n+\half}&=&0\label{m1},\\
-\left(1+\frac{\dt}{2} \exp(\dt/2)\right)A\boldsymbol\beta^{n+\half}+
A\boldsymbol\gamma^{n+\half}&=&\displaystyle\Psi^{n+\half}\label{m2},\\
(M+\frac{\dt^2}{2} J)\boldsymbol\alpha^{n+\half}-
\frac{\dt^2}{2}A_1'{\boldsymbol\gamma}^{n+\half}&=&\Phi^{n+\half},\label{m3}
\end{eqnarray}
where
$$
\Psi^{n+\half}=\frac{\dt}{2} \displaystyle\sum_{i=0}^{n-1}\left[\exp(t_{n+1}-t_{i+1/2})+\exp(t_{n}-t_{i+1/2})\right]
A\boldsymbol\beta^{i+1/2},
$$
and
$$
\Phi^{n+\half}=M(3\boldsymbol\alpha^{n}-\boldsymbol\alpha^{n-1})+
\frac{\dt^2}{2}[A_1'{\boldsymbol\gamma}^{n-\half}
-J {\boldsymbol \alpha}^{n-\half}]+\dt^2 L^{n+\half}.
$$

Then, the unknown vector
$[\boldsymbol\alpha^{n+\half}\,\boldsymbol\beta^{n+\half}\, \boldsymbol\gamma^{n+\half}]$
is the solution of a linear system with a coefficient matrix

\begin{displaymath}
\boldsymbol A=\left[
\begin{array}{ccc}
 A_1& A & J_1 \\
0 & -\left(1+\frac{\dt}{2} \exp(\dt/2)\right)A & A \\
M+\frac{\dt^2}{2} J&0 & -\frac{\dt^2}{2} A_1'
 \end{array}\right],
 \end{displaymath}
and a right hand side
\begin{displaymath}
\boldsymbol b=
 \left[\begin{array}{c} 0\\
 \Psi^{n+\half}\\
 \Phi^{n+\half}\end{array}
 \right].
 \end{displaymath}
The solution will
provide the values of $U^{n+\half}$, ${\bf Q}^{n+\half}$ and $\Z^{n+\half}$ for $n=1,\cdots,N-1$.

\noindent
{\bf Convergence of $\|u(t_n)-U^n\|$ and $\|\bs(t_n)-\Z^n\|$.}  We show the order of convergence in the $L^2$-norm of the error in the flux $\bs$ and in the $L^2$-norm of the error in the velocity $u$. We observe that the optimal order of convergence predicted by our theory (see Table \ref{OC}) is achieved.

\begin{table}[H]
\caption{Order of convergence of $\|e_U(t_N)\|,$ when $t_N=1$.}
 \begin{center}
\begin{tabular}{|c|c|c|c|c|c|c|}
\hline
$C_{11}$ $\rightarrow$&  $O(1)$ & $O(1)$ & $O(1) $  & $O(h^{-1})$ & $O(h^{-1})$  & $O(h^{-1}) $\\
$C_{22}$ $\rightarrow$& 0      & $O(1)$ & $O(h)$ & 0           &   $O(1)$      & $O(h)$ \\
\hline
 p = 1 & 2.1589 & 2.4079 & 2.2915 & 2.3508 & 2.2166   & 2.2752\\
\hline
p = 2  & 3.7140 & 3.4246 & 3.4523 & 3.3009 & 3.4130    & 3.1747\\
\hline
 p = 3 & 4.4037 & 4.0778 & 4.5808 & 4.1506 & 3.4890    & 3.9036\\
 \hline
\end{tabular}
\end{center}
\label{1L}
\end{table}

\begin{table}[H]
\caption{Order of convergence of $\|{\bf e}_{\Z}(t_N)\|$ when $t_N=1$}
\label{1H}
\begin{center}
\begin{tabular}{|c|c|c|c|c|c|c|}
\hline
$C_{11}$ $\rightarrow$&  $O(1)$& $O(1)$ & $O(1)$   & $O(h^{-1})$ & $O(h^{-1})$  & $O(h^{-1}) $\\
$C_{22}$ $\rightarrow$& 0      & $O(1)$ & $O(h)$ & 0           &   $O(1) $     & $O(h)$ \\
\hline
 p = 1 & 1.3815 & 1.1904 & 1.1681 & 1.0763& 1.1414 & 1.0252 \\
\hline
p = 2 &2.3776 & 2.5265 &  2.3054 & 2.3040&2.2181 &  2.3379 \\
 \hline
 p = 3 &3.2914 & 3.3250 &  3.5694 &  3.5365& 3.4465&3.3012   \\
 \hline
\end{tabular}
\end{center}
\end{table}

\noindent
{\bf Example}: Choose $f$ in such a way that the exact solution is $$u(x,y,t)=\exp(t)sin({\pi}{x})sin({\pi}y).$$
We compute the order of convergence for $e_U$ and ${\bf e}_{\Z}$ at $t_N=1$ for the cases $1\le p\le 3$ with different choices of stabilization parameters $C_{11}$ and $C_{22}$. Tables \ref{1L} and \ref{1H} present the computed order of convergence for $\|e_U\|$ and $\|{\bf e}_{\Z}\|$ at $t_N=1$, respectively. In Figures \ref{1u} and \ref{1s}, we present the convergence behavior of $\|e_U\|$ and $\|{\bf e}_{\Z}\|$ at $t_N=1$, respectively with the mesh function $h$ and for $1\le p\le 3$ on uniform triangular meshes when $C_{11}=O(\frac{1}{h})$ and $C_{22}=O(h)$ . We observe that the computed order of convergence match with the predicted order of convergence.

\section{Conclusion}
In this paper, we have proposed and analyzed an $hp$-LDG method for a hyperbolic type integro-differential equation. Compared to the elliptic case \cite{ccps}, \cite{ps}, we have, in this article, established similar $hp$-error estimates for the semidiscrete scheme after suitably modifying the numerical fluxes. Due to the presence of integral term, an introduction of an expanded mixed Ritz-Volterra projection helps to achieve optimal estimates. Further, we have applied a second order implicit method to the semidiscrete scheme to derive a completely discrete scheme and have derived optimal error estimates. Finally, we have also discussed some numerical results.

\vspace{.5cm}

\begin{figure}
    \begin{center}
    \includegraphics*[width=12.0cm,height=6cm]{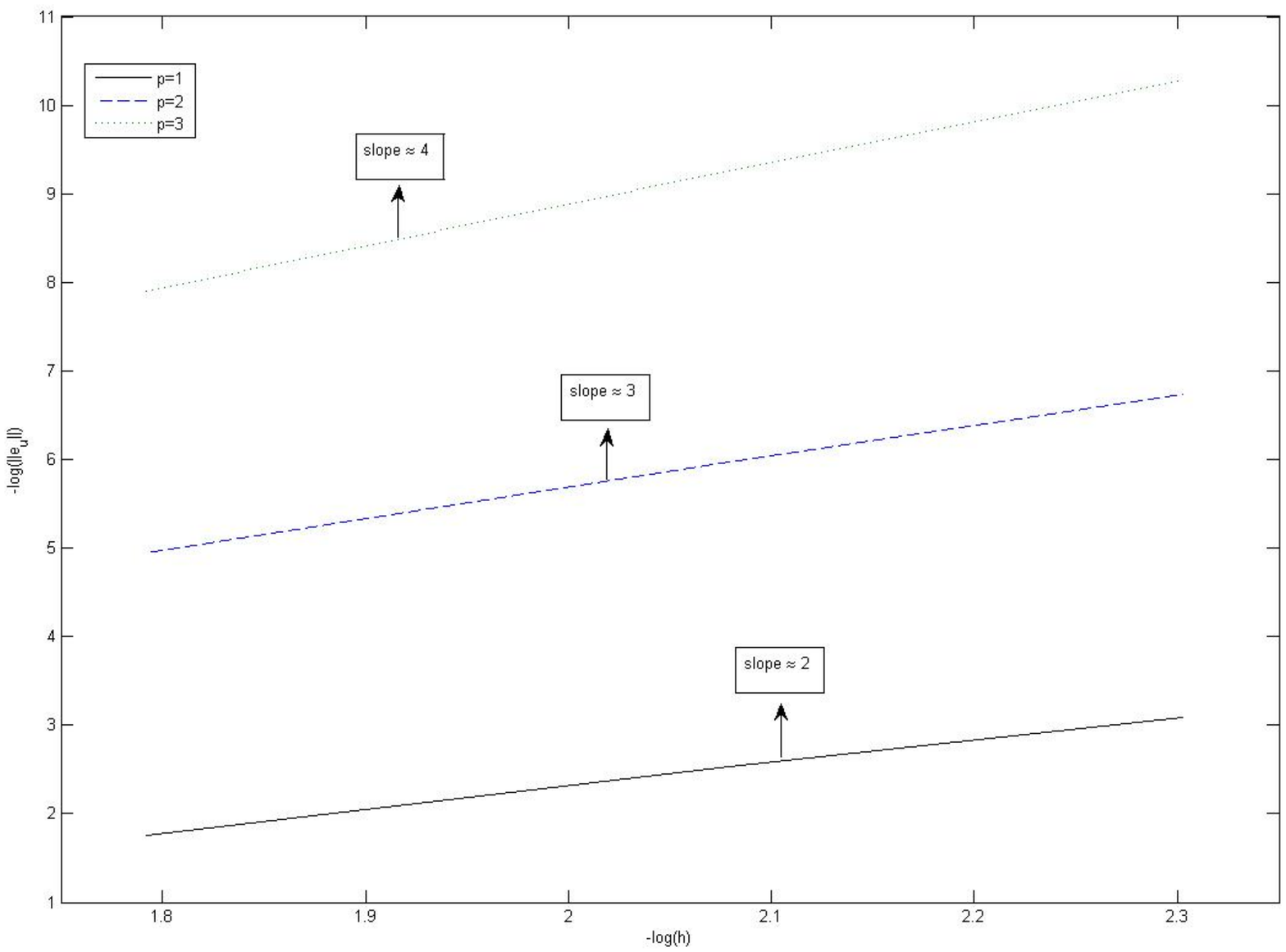}
       \caption{Order of convergence for $\|e_U\|$ at $t_N=1$ when $C_{11}=O(\frac{1}{h})$ and $C_{22}=O({h})$.}
       \label{1u}
    \end{center}
	\end{figure}
	
\begin{figure}
    \begin{center}
   \includegraphics*[width=10.0cm,height=6cm]{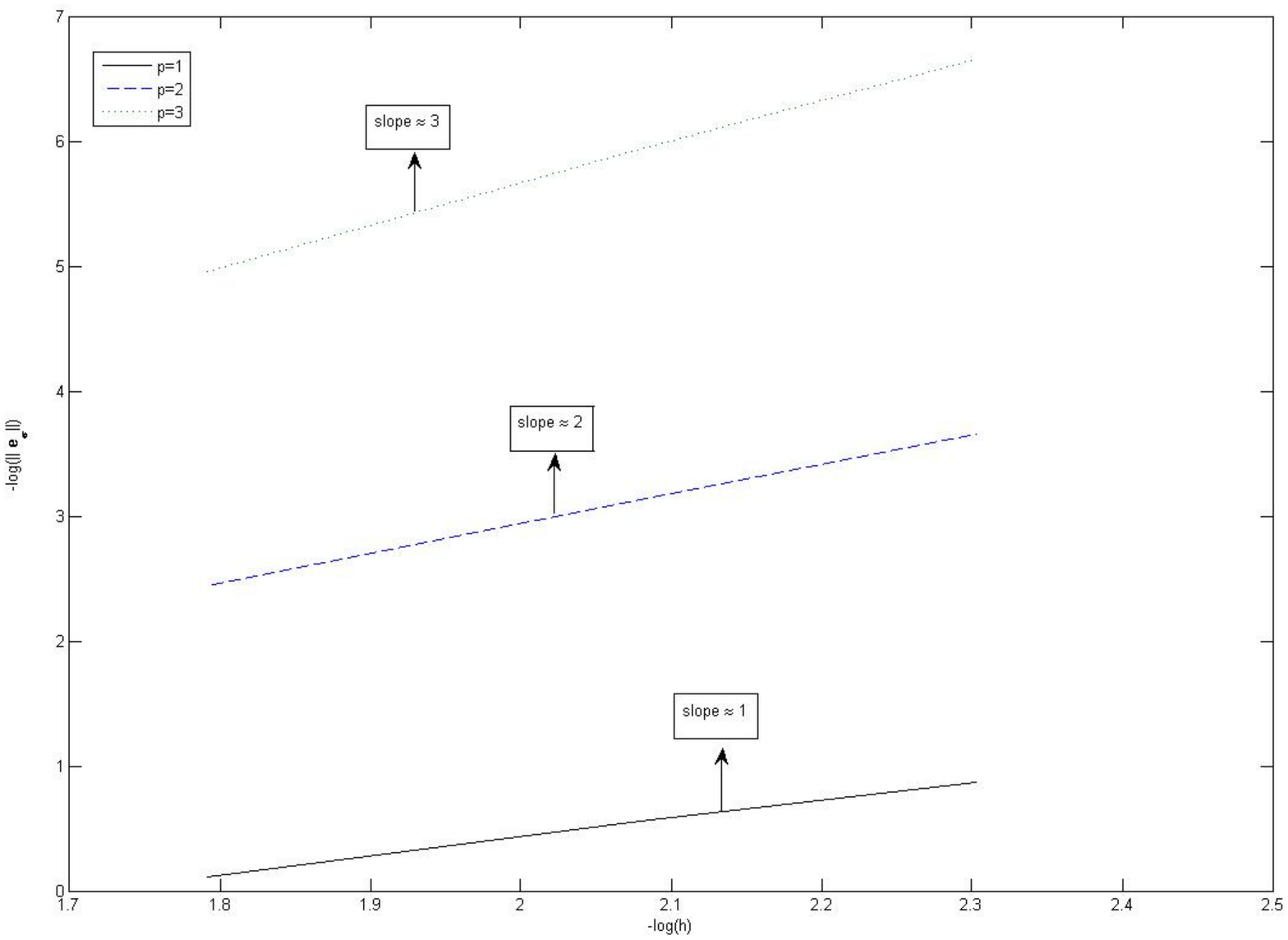}
       \caption{Order of convergence for $\|{\bf e}_{\Z}\|$ at $t_N=1$ when $C_{11}=O(\frac{1}{h})$ and $C_{22}=O({h})$.}
       \label{1s}
    \end{center}
	\end{figure}


\noindent
{\bf Acknowledgment:} The first two authors gratefully acknowledge the research
support of the Department of Science and Technology, Government of India
through the National Programme on Differential Equations: Theory, Computation and Applications vide DST Project No.SERB/F/1279/2011-2012. The first author
acknowledges the support by Sultan Qaboos University under Grant IG/SCI/DOMS/31/02.

\end{document}